\documentclass{amsart}
\usepackage[english]{babel}
\usepackage{amsmath,amsthm}
\usepackage{amsfonts}
\usepackage{graphicx}
\usepackage[colorlinks, linkcolor=blue, anchorcolor=blue,
  citecolor=blue]{hyperref}
\usepackage{amsfonts,amssymb,latexsym,epsfig,amsmath}
\usepackage{filecontents}

\newtheorem{thm}{Theorem}[section]

\newtheorem{lem}[thm]{Lemma}
\newtheorem{prop}[thm]{Proposition}
\theoremstyle{definition}
\newtheorem{defn}[thm]{Definition}
\theoremstyle{remark}
\newtheorem{rem}[thm]{Remark}

\numberwithin{equation}{section}
\newcommand{\norm}[1]{\left\Vert#1\right\Vert}
\newcommand{\abs}[1]{\left\vert#1\right\vert}
\newcommand{\set}[1]{\left\{#1\right\}}
\newcommand{\Real}{\mathbb R}
\newcommand{\Epc}{\mathbb E\,}
\newcommand{\eps}{\varepsilon}

\newcommand{\R}{\mathbb{R}}
\newcommand{\Rd}{\mathbb{R}^d}

\newcommand{\alp}{\alpha}

\newcommand{\sign}{\text{ sign }}
\newcommand{\supp}{\text{ supp }}
\newcommand{\tr}{\text{ tr}}

\DeclareMathOperator{\essinf}{ess\,inf\,}
\DeclareMathOperator{\esssup}{ess\,sup\,}

\makeatletter
\def\l@subsection{\@tocline{2}{0pt}{2.5pc}{5pc}{}}
\makeatother

\begin{document}

\title[$L^1$ semigroup generation for Fokker-Planck operators]
{$L^1$ semigroup generation for Fokker-Planck operators associated with
  general L\'evy driven SDEs}

\author[L. Chen]{Linghua Chen}%
\address{Norwegian University of Science and Technology, NO-7491, Trondheim, Norway}%
\email{linghua.chen@ntnu.no}%

\author[E. R. Jakobsen]{Espen R. Jakobsen}%
\address{Norwegian University of Science and Technology, NO-7491, Trondheim, Norway}%
\email{espen.jakobsen@ntnu.no}%


\thanks{E. R. Jakobsen is supported by the Toppforsk (research
  excellence) project Waves and Nonlinear Phenomena (WaNP), grant
  no. 250070 from the Research Council of Norway.}%
\subjclass{
  47D06,  
  47D07,  
  47G20,  
  35K10,  
  60H10,  
  60G51}  
\keywords{Semigroup generation, elliptic regularity, nonlocal operators, Fokker-Planck equation, Kolmogorov forward equation,
  stochastic differential equation, L\'evy process}%

\date{\today}%
\makeatletter
\def\l@subsection{\@tocline{2}{0pt}{2.5pc}{5pc}{}}
\makeatother

\begin{abstract}
  We prove a new generation result in $L^1$ for a large class of non-local
    operators with non-degenerate local terms. This class contains the
    operators appearing in Fokker-Planck or Kolmogorov forward
    equations associated with L\'evy driven SDEs,~i.e.~the adjoint
    operators of the infinitesimal generators of these SDEs.
    As a byproduct, we also
    obtain a new elliptic regularity result of independent interest.
    The main novelty in
    this paper is that we can consider very general L\'evy operators,
    including state-space depending coefficients with linear growth and general
    L\'evy measures which can be singular and have fat
    tails.
\end{abstract}
\maketitle



\section{Introduction}\label{sec intro}
In this paper we prove an $L^1$ generation result for Fokker-Planck (FP)
or 
Kolmogorov forward operators associated to autonomous L\'evy driven
SDEs. In their most general form such SDEs can be written as (cf. \cite{GS1972,
applebaum2004levy, tankov2003financial, Oksendal2007})
\begin{align}\label{eq sde levy}
        dY_t =&\ b(Y_{t-})dt + \sigma(Y_{t-})dB_t\\ & + \int_{{|z|<1}}
        p(Y_{t-},z)\tilde N(dz,dt) +\int_{{|z|\geq 1}} p(Y_{t-},z)N(dz,dt),\nonumber
\end{align}
where $b: \Real^d \to \Real^d$, $\sigma: \Real^d \to \Real^{d\times
  n}$, $p:\Rd\times\Rd \to \R^{d\times m}$, $B_t$ is a $n$-dimensional Brownian
motion,  and $N$ and $\tilde N$ are
 $m$-dimensional Poisson and compensated Poisson random
measures, respectively. Under suitable assumptions
(cf. \cite{protter2005stochastic}), the solution $Y_t$ of \eqref{eq sde levy}
is a Markov process with infinitesimal generator $L^*$,
\begin{align}\label{eq def of L*}
  L^*f(y)= & \sum_{i=1}^d b_i(y) \partial_i f(y)
    +\frac{1}{2}\sum_{i,j=1}^d a_{ij}(y) \partial_{i}\partial_j f(y)\\
    \notag & + \sum_{k=1}^m \int_{{|z|<1}}
        \big[f(y+p_k(y,z))-f(y)-D f(y) p_k(y,z)\big]\nu_k(dz)\\
    \notag & + \sum_{k=1}^m \int_{{|z|\geq 1}}\big[f(y+p_k(y,z))-f(y)\big]\nu_k(dz),
\end{align}
where $a:=\sigma\sigma^T$, $\nu(dz)dt:=\Epc N(dz,dt)$. For convenience
we assign $\nu(\set{0})=0$, $p=(p_1,\dots,p_m)$, and $\nu=(\nu_1,\dots,\nu_m)$.

In many cases, the process $Y_t$ admits a probability density function (PDF)
$u(t,x)$, a function $u\geq0$ such that 
$\Epc \phi(Y_t)=\int_{\Real^d} \phi(x) u(t,x)dx$ for all $\phi\in C_b(\Real^d)$.
Formally the PDF $u$ solves the Fokker-Planck or forward
Kolmogorov equation
\begin{equation}\label{eq FP}
        \partial_t u(t,x)  =L u(t,x),
\end{equation}
where $L$ is the adjoint of $L^*$. It is this operator $L$ that we
call the Focker-Planck operator.
%
 To be precise, we define $L$
on the domain $D(L):= C_c^{\infty}(\Real^d)$ as the adjoint operator of $L^*$
in \eqref{eq def of L*}, i.e.
\begin{align}
 \notag   L u(x) = &  \frac 12 \sum_{i,j=1}^d  \partial_{ij}\big(a_{ij} u(x)\big)
            - \text{div}\big(b(x)u(x)\big)+ \text{div}\Big(\int_{\set{r\leq |z|<1}}  p(x,z)\nu(dz)u(x)\Big)\\
\label{eq L}  & +\int_{{|z|<r}} \big[u(x-q(x,z))-u(x)+D u(x)q(x,z)\big]  m(x,z) \nu (dz)\\
    \notag & + D u(x)^T \int_{{|z|<r}}\big[p(x,z)-q(x,z) m(x,z)
    \big]\nu (dz) \\
    \notag & + u(x)  \int_{{|z|<r}}\big[ m(x,z) +\text{div}_x
      p(x,z)-1\big]\nu (dz)+ J_r u(x),
\end{align}
for $r>0$ small enough, \ $y(x,z) =x-p(y(x,z),z)=:x-q(x,z)$ \ and \ $m(x,z) :=\det
\big(D_x y(x,z)\big)$, and $J_r$ is the adjoint of
$J_r^* f(y) := \sum_{k=1}^m \int_{|z|\geq r}
[f(y+p_k(y,z))-f(y)]\nu_k(dz).$

To obtain $L^p$ or Sobolev space theories for
such complicated $x$-depending non-local operators, the literature resorts to the global invertibility assumption \cite{garroni2002second,Bally2014},
\begin{align}
    \label{eq global est of 1+p'x}
    0<C^{-1}
    \leq \det\left( 1_d  + D_y p(y,z) \right)\leq C\qquad\text{for
      all}\qquad y,z\in\R^d.
\end{align}
Such an assumption is crucial and e.g. allows one to show  (under some further
assumptions) that $Lu$ belongs to $L^p$ for any
$u\in C_c^\infty$ and $p\in[1,\infty]$, that $L$ is indeed the adjoint
of $L^*$, and that 
$J_r$ then takes the explicit form
(cf. Section 2.4 in \cite{garroni2002second})
\begin{align}\label{def Jr}
          J_r u(x)= &\int_{{|z|\geq r}} [u(x-q(x,z))m(x,z)-u(x)] \nu (dz).
\end{align}
But  assumption \eqref{eq global est of 1+p'x} is very restrictive
and excludes many applications,
including most $x$-depending cases of interest! One of the main
contributions of this paper is to show how it can be dropped
completely, even in the borderline $L^1$ setting.
We will see that we can still work with $L$ even though
e.g.~$J_r$ now will be defined through duality only,
without an explicit representation.

{\em The main result of this paper is that under quite general assumptions,
$L$ generates a strongly continuous contraction
  semigroup on $L^1(\Real^d)$.}

A standard consequence is then that there exists a
unique mild solution in  $L^1$ of the
Cauchy problem for \eqref{eq FP} \cite{engel2006one}, and under
further assumptions, one
can prove that this solution is the PDF of the process $Y_t$ \cite{BKRS,CNJLocal,LC2016}.
Here it is crucial that we work in the
space $L^1$ since PDFs by definition belong to this
space but  in general not to $L^p$ for any $p>1$.
An other application is the
convergence of approximations and numerical methods. Many such results
follow from Kato-Lie-Trotter or Chernoff formulas where the generation
result is a prerequisite \cite{CNJLocal, Butko2016}. In
\cite{CNJLocal} generation is the most difficult step of the proof,
and in many cases, our new generation result provides the generation result
needed in \cite{Butko2016}  (Assumption 6).

The assumptions of our generation result include a uniformly elliptic
local part,
unbounded coefficients with finite differentiability, and general
L\'evy measures (can be singular and have fat tails etc.). In particular the conditions on the
non-local part are very general, covering most jump models in applications
\cite{applebaum2004levy, Barndorff-Nielsen1997, tankov2003financial, schoutens2003levy}.
The restrictive assumption is mainly the uniform ellipticity,
which means the local part can not degenerate/vanish in any
direction. In the literature, such ellipticity or weaker
hypo-ellipticity are typically used to guarantee the existence of
(smooth) PDFs.

The main tools of the proofs are taken from semigroup
theory. We essentially use the Lumer-Phillips theorem to prove
the semigroup generation of dissipative operators in $L^1$.
This is not an easy task.
The difficulty arises not only from the space $L^1$ being
non-reflexive, as we have already encountered in the case without
jumps \cite{CNJLocal},
but also because of the complicated non-local terms
in the FP operator. 
Since we treat very general L\'evy  models and unbounded
coefficients, we can not use the standard global invertibility
assumption \eqref{eq global est of 1+p'x} and show
 semigroup generation directly.
In stead, our strategy is to write the operator as the sum of three
parts that we analyze separately:
the local part, the small jumps part, and the large jumps part.
Through a non-trivial extension of the analysis of \cite{CNJLocal}
(see below), we show that (the sum of) the two first parts generates a
strongly continuous contraction semigroup on
$L^1(\Real^d)$. The presence of the third part is new in this setting
and crucial for the analysis. We show that it is a bounded operator on
$L^1(\Real^d)$ and then treat it as a perturbation to the semigroup
generated by the sum of the other two parts.

Note that in this new approach,
no invertibility assumption is needed. This is true even
though we need invertibility to handle the small jumps term. But
since we have split of the large
jumps, we only need {\em local} invertibility now. By localizing as much as
we need (taking $r$ in \eqref{eq L} small enough), we observe that
 invertibility follows from a standard Lipschitz
assumption on $p$ (cf. Proposition \ref{prop Ir and Jr well def} (b) and proof).
For this argument to work, we also have to handle
the remaining large jumps term using only duality arguments.

A key next step in the generation argument is then to show that the first
and the second parts of the FP operator are dissipative in $L^1$ and that their
corresponding adjoints are dissipative in
$L^\infty$. Both results rely on the negativity of the
corresponding operators. In the $L^1$ setting it translates into
the inequality $\int_{\set{u\neq 0}}L|u|\, dx\leq0$ where
$L$ denotes the FP operator. The proof is technical and involve
separation and approximation of the domains $\{u>0\}$ and $\{u<0\}$
where $|u|$ is smooth. The non-local case is more difficult and requires
additional arguments because the
domains can no longer be separated as in the local case.
On the other hand, to show dissipativity of the adjoint in $L^\infty$,
we first prove that the maximal domain of the adjoint is contained in
certain Sobolev spaces. To this end, we obtain new elliptic regularity
results for non-local operators, extending recent local results in
\cite{ZhangBao2013}. In the local case,
dissipativity then follows from an argument using the Bony maximum
principle for Sobolev functions \cite{CNJLocal}. Here this argument is extended to our
non-local operators using additional ideas from \cite{GimbertLions1984}.

Our elliptic regularity
result is of independent interest: It applies to very general L\'evy
operators, operators with degenerate non-local parts, unbounded
and variable coefficients, and general L\'evy  measures.

Let us now briefly discuss the background setting of our problem.
Over the past decades, there has been a large number of publications in
the field of stochastic dynamics and its various application areas --
including physics, engineering, and finance. In these fields, the
response of dynamical systems to stochastic excitation is studied, and
the typical model is (a system of) stochastic differential equations
(SDEs). Traditionally, the driving noise has been Gaussian, but there
is a large and increasing number of applications that need more
general L\'evy  driving noise like e.g. anomalous diffusions in physics and
biology and advanced market models in finance and insurance
\cite{applebaum2004levy, Barndorff-Nielsen1997, tankov2003financial,
schoutens2003levy, kyprianou2006exotic, mikosch2009non}. A
common feature and difficulty of such models are that the corresponding
processes may have sudden jumps and hence discontinuous realizations
or sample paths.

Then we take a look the literature related to  the  semigroup
generation result. For local forward equations and SDEs driven by Brownian
motion, many classical generation results are given e.g. in
\cite{engel2006one}. More recent results for $L^1$ and
unbounded coefficients can be found in \cite{Fornaro2007747}.
Also for many non-local operators like fractional
Laplacian or generators of L\'evy  processes such generation results are classical,
see e.g. Theorem 3.4.2 in \cite{applebaum2004levy}. That book also gives
generation results in $C_0$ for more complicated generators of L\'evy
driven SDEs in Theorem 6.7.4.
When it comes to generation in $L^1$, we
have only been able to find one paper on non-local operators with
variable coefficients. Theorem 1.1 in \cite{Wang2013} gives such a
result for the operator $L=-(-\Delta)^{\alpha/2}+b(x)\cdot\nabla$.
Note well that these results do not apply to the FP
operator directly, but to its {\em adjoint}. In the local case, the regularity
of the coefficients allows us to rewrite the FP operator as an
adjoint operator plus a  (possibly unbounded) zero-order term. Hence
generation may follow from results for this augmented
``adjoint'' operator as discussed in
\cite{CNJLocal}. However, in the non-local case, this trick is not available
unless we assume also the very restrictive global invertibility
assumption \eqref{eq global est of 1+p'x}.

Generation results can also be obtained in a completely
  different way as a consequence of so-called heat kernal
  analysis. There the aim is
  to obtain sharp bounds on the heat kernals
  or transition probabilities $p(t,x;s,y)$ of the Markov
  process defined by \eqref{eq sde levy}. The semigroup $P_t$
  generated by $L$, can
  then be explicitly defined as $P_tf(x)=\int_{\R^d}\rho(t,x;0,y)f(y)dy$
  for suitable functions $f$.
  This research area dates back to \cite{Ar68}, and more
recently also includes jump processes and non-local operators (e.g. \cite{BL02,BJ07,CK10}).
We will focus on \cite{CHXZ17} which seems to have the most
general results that apply to L\'evy driven SDEs with variable coefficients.
The assumptions include uniform local ellipticity, ``bounded''
coefficients,  and a non-local part that satisfies some moment
condition and is comparable (from one side) to the fractional Laplacian.
In this case $P_t$ is
a strongly continuous contraction semigroup on $L^1$ (and $L^p$ for any
$\rho\in[1,\infty]$) by Theorem 1.1 (1), (2), (5), and (6) of
\cite{CHXZ17} and the application of standard arguments.

Compared with existing results, our generation result applies to FP
operators with much more general jump/non-local parts and
unbounded coefficients. Moreover, we do not use heat kernel analysis,
but rather a direct semigroup approach.

\subsection*{Outline.}
In Section \ref{subsec hypo} we state the assumptions and the main result.
Then we prove our main results in Section \ref{sec:pfm}.
In Section \ref{sec dspt of L} we prove that the generator of the
SDE and its adjoint are dissipative.
Many required properties of the non-local operators are obtained in
Section \ref{sec well-def of L}, including that the long jump part of
the operator is bounded on $L^1$. Finally, Section \ref{sec ell reg} is
 devoted to the proof of the  elliptic
regularity result.

\subsection*{Notation}\label{subsec lay and note}

The following notation will be used throughout the paper:
$\partial_t:= \frac{\partial}{\partial t}$,
$D=D_x:= \left(\frac{\partial}{\partial x_1},\cdots,\frac{\partial}{\partial x_d}\right)^T
            =:(\partial_1 ,\cdots,\partial_d )^T$,$\norm{\cdot}_1:= \norm{\cdot}_{L^1(\Real^d)}$,
$\norm{\cdot}_{\infty}:= \norm{\cdot}_{L^{\infty}(\Real^d)}$,
$\essinf$ is the
essential infimum, $\Epc$ denotes the mathematical expectation;
$1_d$ the identity matrix in $\R^{d\times d}$;
$C_b^k(\Rd)$ and
$C_c^\infty(\Rd)$ the spaces of functions with bounded continuous
derivatives up to $k$-th order and smooth compactly supported
functions, respectively;
$\mathcal D'(\Rd)$ the dual space of $C_c^\infty(\Rd)$.

The following abbreviations are used: PDF - probability density
function, SDE - stochastic differential equation,
FP - Fokker-Planck.

\section{Semigroup generation}
\label{subsec hypo}
In this section, we state the assumptions, our main result on semigroup
generation, a related elliptic regularity result, and remarks. Elliptic
regularity is needed for our proof of generation.
The properties of the operator $L$ and the proof of the generation
result will be given the next section.

We will use the following assumptions:
\begin{itemize}
 \item[\bf (H1)]
    $b\in C^1(\R^d,\R^d)$ and $\sigma\in C^2(\R^d,\R^{d\times n})$,
     and there exists a constant $K>0$ such that for all
    $x\in \Real^d$, $j=1,\dots,n$, and $j,k=1,\dots,d$,
    \[
        |\partial_k\sigma_{ij}(x)|+|\partial_k b_i(x)|  \leq K.
    \]\end{itemize}
For all $k=1,\cdots,m$,
\begin{itemize}
  \item[\bf(H2)] $p_k:\Rd\times \Rd \to \Rd$ is Borel
    measurable, $C^1$ in $y$, and for $\nu$-a.e. $|z|<1$,
    $p_k(\cdot,z)$ is $C^2$ in $y$ and
        \begin{align*}
          & \abs{p_k(y,z)}\leq K(1+|y|)|z|&&\text{for all }y\in\Rd,\\
          & \abs{D_y p_k(y,z)}\leq K|z|&&\text{for all }y\in\Rd,\\
           & \abs{D_y^2 p_k(y,z)}\leq C_R|z|&&\text{for all }|y|\leq R.
        \end{align*}
  \item[\bf(H3)] $\nu_k$ is a non-negative Radon measure satisfying
   \[
        \int_{\Rd} (1\wedge |z|^2)\nu_k(dz)<\infty.
   \]
\end{itemize}

We will also use the following more abstract assumption:
\begin{itemize}
    \item[\bf(E)] (Elliptic regularity)  Let $J_r^* f(y) :=
          \sum_{k=1}^m \int_{|z|\geq r} [f(y+p_k(y,z))-f(y)]\nu_k(dz)$,
          for some $r>0$ small enough. If
        \begin{equation}\label{eq distr sol}
            f,\,g\in
          L^\infty(\Rd)\qquad\text{and}\qquad (L^*-J_r^*)f=g \quad\text{in}\quad \mathcal D'(\Rd),
        \end{equation}
        then  $f\in W^{2,p}_{\text{loc}}(\R^d)$ for some $p>d$.
\end{itemize}

\begin{rem}
Any L\'evy measure $\nu_k$ and most
$p$'s from applications satisfy assumptions (H1) - (H3). E.g. the
$\alpha$-stable processes with $p(y,z)=z$ and $\nu(dz)=\frac{c_\alpha
  dz}{|z|^{d+\alpha}}$, $\alp\in(0,2)$. Unbounded $p$'s appear in
finance and insurance \cite{tankov2003financial, Benth2001OptPort,mikosch2009non, CaiYang2014}, e.g.  $p(y,z)=yz$ and
$p(y,z)=y(e^z-1)$.
The jump term $p$ is allowed
   to vanish on arbitrary large sets, and then the
   non-local part of the FP operator degenerates.
\end{rem}

Assumptions (H1) -- (H3) (except the $C^2$ regularity) are
standard assumptions for the existence and uniqueness of strong
solutions of L\'evy driven SDEs \eqref{eq sde levy}
\cite{applebaum2004levy, Oksendal2007}. They imply that the
coefficients may be unbounded in $y$ (with linear growth), and the
assumptions on the non-local operator are very general indeed: SDEs
with arbitrary L\'evy jump terms, even strongly degenerate ones, are
included. In particular, we do not require any invertibility of
$y+p_k(y,z)$ to define $L$ as the adjoint of the generator
$L^*$ like in \cite{garroni2002second,Bally2014} where the global
assumption \eqref{eq global est of 1+p'x} is used.
Note that this global condition is always satisfied when
$p$ does not depend on $y$, and that this paper is  probably the first
work on semigroup generation not to explicitly or implicitly assume
such a condition.

When it comes to assumption (E), it is most likely already satisfied under
assumptions (H1) -- (H3) if we assume also uniform ellipticity. See
e.g. \cite{CNJLocal} for local
operators. The general case seems not be covered in the literature, so
we will prove that (E) holds under ellipticity and mild additional
assumptions below.

Now we can state the main result of this paper:

\begin{thm}[Semigroup generation]\label{thm L gen s.g.}
  Assume (H1) -- (H3) and (E). Then the closure of $L$ generates a
  strongly continuous contraction semigroup on $L^1(\Real^d)$.
\end{thm}

We now give results verifying assumption (E) under uniform ellipticity
and mild additional assumptions on the jump-terms:
\begin{itemize}
  \item[\bf(HE1)]
   There exists $\alpha>0$ such that for all $x, y\in \Real^d$,
    \[
        y^T a(x) y \geq \alpha |y|^2.
    \]
  \item[\bf (HE2)] There exists some $s\in [1,2)$ such that
    $\int_{|z|<1}|z|^s \nu_k(dz)<\infty$.

\end{itemize}
For all $k=1,\cdots,m$,
\begin{itemize}
      \item[\bf(H2')] (H2) holds, and $p_k(\cdot,z)$ is $C^3$ in $y$ for
        $\nu_k$-a.e. $z$, and there exists
        $\tilde p_k(z)\geq 0$ such that for all $R>0$, $|y|\leq R$,
        and $\nu_k$-a.e. $z$,
         $$\abs{p_k(y,z)}+\abs{D_y p_k(y,z)}+\abs{D_y^2 p_k(y,z)}+\abs{D_y^3 p_k(y,z)}
         \leq C_R\big( |z|\wedge \tilde p_k(z)\big).$$
   \item[\bf(H3')] (H3) holds, and
        $\tilde C:= \max_{k=1,\dots,m}\int_{|z|\geq 1} \tilde p_k(z) \nu_k(dz)<\infty$.
\end{itemize}

Under (HE2), $s<2$ is the maximal (pseudo) differential order of the
non-local part of the FP operator. Since the bound is only from above, the
L\'evy measures $\nu_k$ may be degenerate.

When the L\'evy measure is not too singular ($s=1$) we only need
(HE1) and (HE2). In the general case all assumptions are needed.
\begin{thm}[Elliptic regularity]\label{prop ell reg 01}
Assumption (E) holds if either one of the two sets of assumptions
below hold:
\vspace{-0.1cm}
\begin{itemize}
    \item[(a)] (H1), (H2), (H3), (HE1), and (HE2) with $s=1$.
    \medskip
    \item[(b)] (H1), (H2'), (H3'), (HE1), and (HE2) with $s\in (1,2)$.
\end{itemize}
\end{thm}
    This result will be proved in Section \ref{sec ell reg}.
\begin{rem}
  {\bf (a)}
    If $p\equiv0$ and the operator is local, then Theorem \ref{thm L gen
      s.g.} has been proven in \cite{CNJLocal} with assumption (HE1)
    replacing assumption (E).
%

  {\bf (b)}
    To do our generation proof (to prove the dissipativity of $L^*$) we
    need enough regularity for the
    equation $L^*u=f$ to hold a.e. for any $f\in L^\infty$ and some version
    of the Bony maximum principle to apply. This is encoded in (E), and such a
    condition can only be true under some sort of non-degeneracy conditions
    on the second order local terms (e.g. (HE1)).

    {\bf (c)} Assumption (E) can be relaxed when there are no second-order terms
    in the operator. Then the principal non-local term must be
    non-degenerate. To extend our proofs in this direction, new Bony type
    maximum principles are needed for fractional Sobolev spaces. We will not
    pursue this idea in this paper.

    {\bf (d)}     Non-degeneracy conditions like (HE1) or weaker  H\"ormander
    conditions, along with smoothness assumptions on the coefficients, are
    standard assumptions in the literature to ensure the
    existence of (smooth) PDFs for \eqref{eq sde levy}, see e.g.
    \cite{hiraba1992, Kolokoltsov2000,Komatsu2001,Bodnarchuk2008cond,
    Cass20091416, zhang2014} and references therein.


    {\bf (f)} Elliptic regularity results are well-known for local
    operators (and PDEs), and results that cover the local part of our
    operators ($L^*$
    with $p\equiv0$) can be found in the recent paper
    \cite{ZhangBao2013}. Theorem \ref{prop ell reg 01} (a) is essentially a corollary
    of results in \cite{ZhangBao2013}
    where the non-local term is treated as a lower-order
    perturbation. Part (b) is much more complicated and requires
    additional regularity on $p(x,z)$.

    There are also very general results for
    pseudo-differential operators.
    These results require that the symbols
    are smooth and satisfy certain decay assumptions which are not
    in general satisfied by the operators we consider here, see e.g. Section
    7.3.1 in \cite{abels2012pseudodifferential}.
\end{rem}

In the rest of the paper we set $m=1$ and $p_k(x,z)=p(x,z)$ to
simplify the notation. The general case is similar and will be
omitted.

\section{Properties of $L$ and proof of generation}
\label{sec:pfm}
In this section, we show that $L$ is well-defined and dissipative in
$L^1$, that $L^*$ is dissipative in $L^\infty$, and use a version of
the Lumer-Phillips theorem along with a perturbation result to show
semigroup generation for $L$ in $L^1$.

To work with $L$, we decompose it along with $L^*$ into three
parts. For any $r\in(0,1)$,
\begin{align*}
  L^*= A_r^*+I_r^*+J_r^*,
\end{align*}
where
\begin{align*}
  A_r^* f(y) & = b^T(y) D f(y)+\frac{1}{2}\sum_{i,j=1}^d a_{ij} \partial_{i}\partial_j f(y)+[D f(y)]^T \int_{\set{r\leq |z|<1}} p(y,z)\nu(dz),\\
  I_r^* f(y) & =\int_{\set{|z|<r}} [f(y+p(y,z))-f(y)-[D f(y)]^T p(y,z)]\nu(dz),\\
  J_r^* f(y) & = \int_{\set{|z|\geq r}} [f(y+p(y,z))-f(y)]\nu(dz).
\end{align*}

By integration by parts and the change of variables $x=
y+p(y,z)$ (assuming it is invertible), it follows
that the adjoint
\[
    L=A_r+I_r+J_r,
\]
where
\begin{align}
\nonumber A_r u(x) =\, &  \frac 12 \sum_{i,j=1}^d  \partial_{i}\partial_{j}(a_{ij} u(x))
            - \text{div}\left[ \left(b(x)+ \int_{\set{r\leq |z|<1}}
                p(x,z)\nu(dz)\right)u(x)\right]\\ \label{eq I}
  I_r u(x) = & \int_{{|z|<r}} [u(x-q(x,z))-u(x)+D u(x)q(x,z)]  m(x,z) \nu (dz)\\
    \notag & + (D u(x))^T \int_{{|z|<r}}\left[p(x,z)-q(x,z) m(x,z) \right]\nu (dz) \\
    \notag & + u(x)  \int_{{|z|<r}}\left[ m(x,z) +\text{div}_x
      p(x,z)-1\right]\nu (dz),
\end{align}
for \ $y(x,z) =x-p(y(x,z),z)=:x-q(x,z)$ \ and \ $m(x,z) :=\det
\big(D_x y(x,z)\big)$. The derivation can be found in Section 2.4 in
\cite{garroni2002second}.
If we assume  global invertibility of $y\mapsto y+p(y,z)$,
assumption \eqref{eq global est of 1+p'x}, then $J_r$ has the
explicit form \eqref{def Jr}. One contribution of this paper is to
relax this condition, and not
work with a $J_r$ given by an explicit formula, but rather defined
only by the duality $J_r=(J_r^*)^*$. Moreover, without global invertibility, the
derivation of $I_r$
from $I^*_r$ only holds for $r$ small enough. In this case, we still get
the (local) invertibility needed to do the above-mentioned change of
variables (see Proposition \ref{prop Ir and Jr well def} and Section \ref{sec well-def of L}).

Note that $A_r,I_r,A_r^*,I_r^*$ are unbounded operators while $J_r$ and $J_r^*$
are bounded.
\begin{rem}\label{rem B and L measurable}
    $J_r$ and $J_r^*$ can be defined on $L^\infty$ and $L^1$
    respectively (see below). To make the integrands
    well-defined (Borel or $\nu$-measurable) for functions in $L^1$
    and $L^\infty$, we always work with Borel representatives
    (cf. Remark 2.1 in \cite{AlibaudCJ2012}).
\end{rem}

Now we show that our operators are well-defined on $L^1$.
\begin{prop}\label{prop Ir and Jr well def}\

   \textbf{(a)}   Assume (H1) and $r>0$. Then $A_r$: $D(L)\to L^1(\Rd)$ is
   well-defined.

    \textbf{(b)} Assume (H2) and (H3). Then there is $r_0<\frac 1{4dK}$
    such that $I_r:
    D(L)\to L^1(\Rd)$ is well-defined for all $0<r<r_0$.

    \textbf{(c)} Assume (H2) and (H3) and $r>0$. Then
          $J_r:L^1(\Rd)\to L^1(\Rd)$ is well-defined and bounded,
    $$\norm{J_r}\leq 2\nu\left(\set{|z|\geq r}\right).$$
\end{prop}
It follows that $L:D(L)\to L^1(\R^d)$ is well-defined if (H1) -- (H3) holds.
The proof will be given in Section \ref{sec well-def of L}.

Next, we show that the operators (and their adjoints) are dissipative
in the sense of the following definition (see e.g. Section II.3 of \cite{engel2006one}):

\begin{defn}
    A linear operator $(B,D(B))$ on a Banach space $(\mathbb  X, \norm{\cdot})$ is
    \emph{dissipative} if \ $\norm{(\lambda-B)u}\geq \lambda\norm{u}$ \ for all $\lambda>0$ and all $u\in D(B)$.
\end{defn}

\begin{thm}\label{prop Ar+Ir and adj dspt}\

  \textbf{(a)} Assume (H1) -- (H3) and $r<r_0$, where $r_0$ is defined in
  Proposition \ref{prop Ir and Jr well def}.
    Then $A_r+I_r$ is dissipative on $D(L)\subset L^1(\Rd)$.\smallskip

  \textbf{(b)} Assume (H1) -- (H3), and (E).
    Then $A_r^*+I_r^*$ is dissipative on
    $D(A_r^*+I_r^*)\subset L^{\infty}(\Real^d)$.

  \textbf{(c)} Assume (H2) and (H3). Then $J_r$ is dissipative on
  $L^1(\Real^d)$ for any $r>0$.
\end{thm}
$D( A_r^*+I_r^*)$ is the maximal domain of $A_r^*+I_r^*$ on
$L^\infty(\Rd)$. It will be characterized in Section
\ref{subsec dspt of Ar*+Ir*}.
The proposition is proved in Sections \ref{subsec dspt of Ar+Ir} --
\ref{subsec dspt of Jr}. These proofs and the proofs of related auxiliary
results constitute the main technical innovation of this paper. They
are highly non-trivial, and the PDE-inspired way of doing the proofs
seems to be unconventional.

\begin{rem}
    One can easily check that also $J_r^*$ is dissipative on $L^{\infty}(\Rd)$.
    Hence both $L$ and $L^*$ are dissipative by Section III.2 in
    \cite{engel2006one}.
\end{rem}

We are in a position to use the Lumer-Phillips theorem to prove the
following preliminary generation result.
\begin{prop}\label{prop Ir+Ar gen sg}
  Assume (H1) -- (H3), (E), and $r<r_0$, where $r_0$ is defined in
  Proposition \ref{prop Ir and Jr well def}. Then $A_r+I_r$
  generates a strongly continuous contraction semigroup on $L^1(\Real^d)$.
\end{prop}
\begin{proof}
  Since  $A_r+I_r$ and $A_r^*+I_r^*$ are dissipative by Proposition
  \ref{prop Ar+Ir and adj dspt},
  and $A_r+I_r$ is densely defined ($D(L)$ is dense in $L^1(\R^d)$),
  $A_r+I_r$ generates a strongly continuous contraction semigroup $P_t$ on $L^1(\Real^d)$ by a
  version of the Lumer-Phillips theorem -- see Corollary II.3.17 in
  \cite{engel2006one}.
\end{proof}

To get a generation results for the full operator $L$, we view it
as a bounded perturbation of $A_r+I_r$ and use the following result:
\begin{thm}[Theorem 3.3.4 in \cite{pazy1992semigroups}]\label{thm dspt ptb pazy}
  Let $B_1$ generate a contraction semigroup
  on a Banach space $(\mathbb  X,\norm{\cdot})$ and $B_2$ be dissipative.
  Assume $D(B_1)\subset D(B_2)\subset \mathbb X$ and there is a $\lambda>0$ such that
  \[
    \norm{B_2 x}\leq \norm{B_1 x}+ \lambda \norm{x}\qquad\text{ for all }
        x\in D(B_1).
  \]
  If $B_2^*$, the adjoint of $B_2$, is densely defined,
  then the closure of $B_1+B_2$ generates a strongly continuous contraction semigroups.
\end{thm}

\begin{proof}[Proof of Theorem \ref{thm L gen s.g.}]
   Take $r<r_0$ where $r_0$ is defined in
  Proposition \ref{prop Ir and Jr well def}. Then note that
  $L-J_r=A_r+I_r$ generates   a strongly continuous contraction
  semigroup on $L^1(\Real^d)$
  by Proposition \ref{prop Ir+Ar gen sg}, that
  $J_r$ is bounded and
  dissipative on $L^1(\Real^d)$ by Propositions
  \ref{prop Ir and Jr well def} (c) and
  \ref{prop Ar+Ir and adj dspt} (c), and $J^*_r$ is bounded and hence
  an everywhere defined operator on $L^\infty = (L^1)^*$ by definition.
  Hence the result follows from Theorem \ref{thm dspt ptb pazy}
  with  $B_1=L-J_r$, $B_2=J_r$, and $\lambda\geq \norm{J_r}$.
\end{proof}

\section{Operators in $L^1$ -- proof of Proposition \ref{prop Ir and Jr well def}}
\label{sec well-def of L}
In this section prove Proposition \ref{prop Ir and Jr well def},
i.e. we show that the operators $A_r$ and $I_r$ are well-defined from
$D(L)$ into $L^1$ and $J_r$ well-defined and bounded on $L^1$. For $A_r$ this is
immediate from the definition of this operator, so we will focus on
the other two operators.
If we assume the global invertibility \eqref{eq global est of 1+p'x}, then
the results follow from arguments similar to those given in Section
2.4 of \cite{garroni2002second}. However, the general case is more
complicated and will be dealt with now.

We recall from Section \ref{sec intro} that
$$y(x,z) =x-p(y(x,z),z)=:x-q(x,z)\quad\text{and}\quad m(x,z) :=\det
\big(D_x y(x,z)\big),$$ and note that by the implicit function
theorem
    \begin{align}\label{m-eq}
        m(x,z)=\det \big( 1_d - D_x q(x,z) \big)
            =\frac{1}{\det \big( 1_d + (D_y p)(y(x,z),z) \big)}.
    \end{align}

\subsection{Proposition \ref{prop Ir and Jr well def} (b) -- the
  operator $I_r$}\label{subsec well def of I_r}
Throughout this section, we assume
(H2) -- (H3) with $r<1/(4dK)$, and we define the set
\[
    U_r:=\Real^d\times\set{z\in \Real^d: |z|<r}.
\]
Before we prove the result, we give a long list of technical results.

\begin{lem}\label{lem bound of y(x,z)}
   $|y(x,z)|\leq 2|x|+1$ for $(x,z)\in U_r$.
\end{lem}
\begin{proof}
    Observe that for $(x,z)\in U_r$ with $r<1/(4dK)$,
    \begin{align*}
      |y(x,z)| & \leq |y(x,z)+q(x,z)|+|q(x,z)| \\
                & = |y(x,z)+q(x,z)|+|p(y(x,z),z)|\\
                & \leq |x|+K(1+|y(x,z)|)|z|\\
                & \leq |x|+\frac{1}{4d}(1+|y(x,z)|)\\
                & \leq |x|+\frac{1}{2}(1+|y(x,z)|),
    \end{align*}
and the result follows.
\end{proof}

\begin{lem}\label{lem est of q(x,z)}
    For some $C>0$ and all $(x,z)\in U_r$,
    \begin{equation*}
      \abs{q(x,z)}\leq C(1+|x|)|z|.
    \end{equation*}

\end{lem}
\begin{proof}
    Just note that
    \[
      |q(x,z)|  =|p(y(x,z),z)|
       \leq K(1+|y(x,z)|)|z|
       \leq C(1+|x|)|z|
    \]
    by (H2) and Lemma \ref{lem bound of y(x,z)}.
\end{proof}

Next we show that invertibility \eqref{eq global est of 1+p'x}
holds if we restrict to the set $U_r$ (compare with (2.2.7) in
\cite{garroni2002second}).
\begin{lem}\label{lem bound of m(x,z)}
    There is $C>1$ such that for all $(x,z)\in U_r$, \eqref{eq global est of 1+p'x}
    holds and hence
$$0<C^{-1} \leq m(x,z)\leq C\quad\text{ for all }\quad (x,z)\in U_r.$$
\end{lem}

\begin{proof}
    Straightforward by definition, assumptions, and \eqref{m-eq}.
\end{proof}

\begin{lem}\label{eq est div(p-q)}
    Define
    $f(\cdot,z):=(\text{div}_y p)(\cdot,z)=\sum_{k=1}^d f_k (\cdot,z)$.
    Then
    \begin{equation*}
      \abs{f(x,z)- f(y(x,z),z)}\leq C_R |z|^2,
        \text{ for all }|x|\leq R\text{ and }|z|<r.
    \end{equation*}
    where $C(x)>0$ locally bounded with respect to $x$.
\end{lem}
\begin{proof}
    Observe
    \[
        f(x,z)- f(y(x,z),z)=
            \sum_{k=1}^d [f_k(x,z)-f_k(y(x,z),z)].
    \]
    For each $k$,
    \begin{align*}
       f_k(x,z)-f_k(y(x,z),z)  & = f_k(x,z)-f_k(x-q(x,z),z)\\
       & =q^T(x,z) (Df_k)(x-\theta q(x,z),z).
    \end{align*}
    By Lemma \ref{lem est of q(x,z)} and (H2),
    \begin{align*}
        \abs{q^T(x,z) (Df_k)(x-\theta q(x,z),z)}  \leq C(1+|x|)|z|K|z|
        =:C(x)|z|^2.
    \end{align*}
    The result follows since
    $\abs{f_k(x,z)-f_k(y(x,z),z)}\leq C(x) |z|^2$.
\end{proof}

\begin{lem}\label{lem decomps 1+M}
  Let $M=M(x,z):=(D_y p)(y(x,z),z)$, then
  \begin{equation*}
    \det(1_d +M)=1+\text{tr}(M)+P(x,z),
  \end{equation*}
    where $\abs{P(x,z)}\leq C|z|^2$ for all $x\in\R^d$.
\end{lem}
\begin{proof}
  This is easily seen by the definition of determinant, the definition of $M$
  and assumption (H2).
  One can also refer to Section 2 of \cite{Brooks2006}. The constant $C$ is uniform in $x$ by (H2).
\end{proof}

\begin{lem}\label{lem bound 1-m(x,z)}
    There exists $C>0$ such that
  \begin{equation*}
    \abs{m(x,z)-1}\leq C|z| \qquad\text{for all}\qquad (x,z)\in U_r.
  \end{equation*}
\end{lem}
\begin{proof}
    Denote
    \[
        D_y(y+p(y,z))=1_d +D_y p(y,z)=:1_d+M.
    \]
    Then by definition
     \begin{align*}
          \abs{m(x,z)-1} & =\abs{\frac{1}{\det(1_d +M)}-1}
      = \abs{\frac{1-\det(1_d +M)}{\det(1_d +M)}}. 
    \end{align*}
    By Lemma \ref{lem decomps 1+M} and (H2),
    $\det(1_d +M)=1+\text{tr}(M)+P(x,z)$, and then for 
    $|z|<r$,
    \[
       |\det(1_d +M)-1|\leq |\text{tr}(M)+P(x,z)|\leq C|z|.
    \]
    The proof is then complete if we can get a lower bound on $|\det(1_d
    +M)|$.

    We claim that $\det(1_d +M)\geq 2^{-d}$. For each entry in $M$ and $|z|<r=\frac{1}{4dK}$,
    \[
        \abs{(\partial_{y_i} p_j)(y,z)}\leq K|z|\leq \frac{1}{4d},
    \]
    and the matrix $1_d +M$ is diagonally dominant.
    By Theorem 1 in \cite{price1951},
    \[
        \det(1_d +M)\geq \prod_{i=1}^d(|\alpha_{ii}|-\beta_i),
    \]
    where $\alpha_{ii}=1+(\partial_{y_i}p_i)(y,z)$ and
    $\beta_i=\sum_{j=i+1}^d\abs{(\partial_{y_i}p_j)(y,z)}$.
    Hence $\abs{\alpha_{ii}} \geq 1-\frac{1}{4d}$ and $\beta_i\leq \frac{1}{4}$, and thus
        $ |\alpha_{ii}|-\beta_i\geq 1-\frac{1}{4d}-\frac{1}{4}\geq \frac{1}{2}$
     and $\det(1_d +M)\geq 2^{-d}.$ The proof is complete.
\end{proof}

\begin{proof}[Proof of Proposition \ref{prop Ir and Jr well def} \textbf{(b)}]
    First note that since $r<\frac1{4dK}$ and the supports of $u$ and
    $I_r$ are compact, we only need to consider $x,z$ on compact sets
    depending on $u$ (and $p$, $q$ -- see below) but not on $r$. With this in mind we
    bound the different terms $I_r$, see \eqref{eq I}. For the first integral,
    \begin{align*}
       & \abs{\int_{{|z|<r}}[u(x-q(x,z))-u(x)+q(x,z)D u(x)]  m(x,z)\nu
         (dz)} \\
            &= \abs{\int_{{|z|<r}} \int_0^1 (1-\theta)q^T(x,z) [D^2 u(x)] q(x,z) d\theta  m(x,z)\nu (dz)}\\
      & \leq  C(u)\int_{{|z|<r}} |z|^2\nu (dz)<\infty.
    \end{align*}
    Here we also used that $q$ is bounded on compact sets. For the second
    integral, we keep in mind that $x=y+p(y,z)$. The integrand is then
    \begin{align*}
        & p(x,z)-q(x,z) m(x,z) \\
       &= p(x,z)-q(x,z)+q(x,z)\left(1- m(x,z)\right) \\
       &= p(x,z)-p(x-q(x,z),z)+q(x,z)\left(1- m(x,z)\right)\\
       &= (D_yp)(x-\theta q(x,z),z) q(x,z) +q(x,z)\left(1- m(x,z)\right)\\
       &= q(x,z)[ (D_yp)(x-\theta q(x,z),z)+ (1- m(x,z))].
    \end{align*}
    Hence by (H2) and Lemmas \ref{lem est of q(x,z)} and
    \ref{lem bound 1-m(x,z)}, for $x,z$ in the compact,
    \begin{align*}
        \abs{p(x,z)-q(x,z) m(x,z)}
            \leq  C(u)|z|(|z|+|z|),
    \end{align*}
    and hence
    \begin{align*}
       & \abs{\int_{{|z|<r}} (D u(x))^T\left[p(x,z)-q(x,z)
        m(x,z)\right]\nu (dz)} \leq C(u)\int_{{|z|<r}} |z|^2 \nu(dz).
    \end{align*}

    For the third integral,
    we take $f(\cdot,z):=(\text{div}_y p)(\cdot,z)$, and note that integrand
    \begin{align*}
       &  m(x,z)+\text{div}_x p(x,z)-1  = m(x,z)+f(x,z)-1 \\
       & = [ m(x,z)+f(y(x,z),z)-1]
            +[f(x,z)- f(y(x,z),z)].
    \end{align*}
    The last term can be estimated by Lemma \ref{eq est div(p-q)},
    \[
        |f(x,z)- f(y(x,z),z)|\leq C(u)|z|^2.
    \]
    For the first term, recall that $M=M(x,z)=(D_y p)(y(x,z),z)$ and note that
    $\text{tr}(M)=f(y(x,z),z)$.
    Then by Lemma \ref{lem decomps 1+M},
    \begin{align*}
        m(x,z)+f(y(x,z),z)-1
      & =\det(1_d+M)^{-1}+\text{tr}(M)-1\\
      & = \frac{1}{1+\text{tr}(M)+P(x,z)}+\text{tr}(M)-1\\
      & =\frac{-P(x,z)+\text{tr}(M)P(x,z)+(\text{tr}(M))^2}{1+\text{tr}(M)+P(x,z)}.
    \end{align*}
    By Lemma \ref{lem decomps 1+M} again,
    \begin{align*}
      \abs{-P(x,z)+\text{tr}(M)P(x,z)+(\text{tr}(M))^2}  \leq C|z|^2\quad\text{and}\quad|M|+|P(x,z)|\leq C|z|
    \end{align*}
    where $C$ does not depend on $x$ and hence
    the support of $u$. We may therefore take a sufficiently small $r_0<\frac1{4dK}$ (independently
    of $u$) such that for $|z|<r<r_0$,
    $$1+\text{tr}(M)+P(x,z) \geq \frac 12.$$

    Hence $\abs{m(x,z)+f(y(x,z),z)-1}\leq C|z|^2$, and it follows that
    the third integral in \eqref{eq I} is well defined.

    From the above estimtates and the compactness of the support, it then
    follows that  there is $r_0>0$ such that $\|I_ru\|_1=\int_{\Rd}
    |I_r u(x)|dx<\infty$ for any $0<r<r_0$ and any $u \in D(L)$. The proof
    is complete.
\end{proof}

\begin{rem}
    In the proof of Lemma 2.4.3 in  in \cite{garroni2002second}, the
    authors claim that if $M$ is a symmetric matrix such that $\det(1_d
    +M)\neq 0$, then
        \[
            \abs{\frac{1}{\det(1_d +M)}-1-\text{tr}(M)}\leq C\norm{M}^2.
        \]
    If we could take $M=D_y p(y,z)$ in this inequality, it would simplify
    our proofs. However, in our setting $D_y p(y,z)$ is not symmertric in general.
\end{rem}

\subsection{Proposition \ref{prop Ir and Jr well def} (c) -- the operator $J_r$}
\label{subsec well-def Jr} We start by two auxilliary results.
\begin{lem}\label{lem J_r u sig-add}
    Assume (H3) and $u\in L^1(\R^d)$. Then $J_ru$ can be represented
    by a bounded, absolutely continuous, and finitely additive signed
    measure $\lambda_u$ such that
    \begin{align}\label{lambda_u}\langle J_ru,
        f\rangle=\int_{\Rd} f(x) \lambda_u(dx) \qquad \text{for all}\qquad
        f\in L^\infty(\R^d).
    \end{align}
    Moreover, the total variation norm \ $|\lambda_u|(\R^d)
    \leq 2\|u\|_{L^1}\nu(\{|z|>r\}).$
\end{lem}
\begin{proof}
    This is quite standard. By the definition and (H3), $J_r^*$ is a
    bounded linear operator on $L^{\infty}(\Real^d)$, and $\norm{J_r^*
    }\leq 2\nu (\set{|z|\geq r})$ since $\abs{J_r^* f(y)}\leq 2\norm{f}_{\infty}
                \nu (\set{|z|\geq r})$ for all $y\in \Real^d$.
    Hence its adjoint operator $J_r$ is a bounded linear operator
    on the dual space of $L^{\infty}(\Real^d)$ with
    $\norm{J_r}=\norm{J_r^* }$, cf. Theorem 3.3 in
    \cite{schechter2001principles}. Hence also $\norm{J_r}\leq 2\nu
    (\set{|z|\geq r})$, and since $\|u\|_{(L^\infty)'}=\|u\|_{L^1}$
    for $u\in L^1$, we have
    $$\norm{J_ru}_{(L^\infty)'}\leq \|J_r\|\|u\|_{(L^\infty)'}\leq
    2\|u\|_{L^1}\nu (\set{|z|\geq r}).$$ Then by Theorem IV.8.16 in
    \cite{schwartz1958linear},
    there is an isometric isomorphism between the dual of
    $L^{\infty}(\Real^d)$ and the bounded, absolutely continuous,
    finitely additive signed (ACFAS) measures. That is, $J_ru\in
    \big(L^\infty(\Rd)\big)'$ corresponds uniquely to a ACFAS measure
    $\lambda_u$ such that \eqref{lambda_u} holds.
    The integral is here defined in the standard way by
    first defining it for finitely(!) valued simple functions and then
    take the limit of total variation. The isometry part of the result
    means that the norm of $J_ru$ equals the total variation
    of $\lambda_u$, $\|J_ru\|_{(L^\infty)'}=|\lambda_u|(\R^n).$ The proof is
    complete.
\end{proof}

We will need the following version of the Radon-Nikodym Theorem.
\begin{thm}[Theorem 10.39 in \cite{wheeden1977measure}]\label{thm Radon-Nkd}
    Let $\mu$ be a (finite and countably) additive set function. If $\mu$ is
    absolutely continuous with respect to the Lebesgue measure,
    then there exists an integrable function $w\in L^1(\Real^d,dx)$,
    such that
    \[
        \mu(E)=\int_E w(x)dx\quad\text{ for all }E\in \Sigma,
    \]
    where $\Sigma$ is the $\sigma$-algebra of all Lebesgue measurable sets.
\end{thm}

\begin{proof}[Proof of Proposition \ref{prop Ir and Jr well def} \textbf{(c)}]\

    \noindent 1.\quad By Lemma \ref{lem J_r u sig-add}, $J_ru$ can be
    represented by a
    bounded, absolutely continuous, and finitely additive signed measure
    $\lambda_u$ such that for any measurable set $E\subset \Rd$,
    \begin{equation}\label{eq iso meas and op}
        \abs{\lambda_u}(E)<\infty\quad\text{ and }\quad
            \lambda_u(E) =\int_{\Rd} \chi_E(x)\lambda_u(dx)
            =\langle J_ru, \chi_E\rangle.
    \end{equation}

    \noindent 2.\quad We check that $\lambda_u$ is in fact also countably additive.
    Suppose $\set{A_k \subset \Rd:\, k=1,2,3,\dots}$ is a sequence of pair-wise
    disjoint Lebesgue measurable sets.
    Then $\chi_{\cup_k A_k}(y)=\sum_k \chi_{A_k}(y)$,
    and by \eqref{eq iso meas and op},
    \begin{align*}
      \lambda_u(\cup_k A_k) & = \langle J_r u,\,\chi_{\cup_k
        A_k}\rangle
            = \langle u,\,J_r^* \chi_{\cup_k A_k}\rangle \\
       & = \int_{\Rd}u(x) \int_{|z|\geq r}\big[\chi_{\cup_k
         A_k}(x+p(x,z))-\chi_{\cup_k A_k}(x)\big] \nu(dz)dx\\
       & = \int_{\Rd}u(x)\int_{|z|\geq r}\sum_k \big[\chi_{A_k}(x+p(x,z))-\chi_{A_k}(x)\big] \nu(dz)dx\\
       & = \sum_k \int_{\Rd}u(x)\int_{|z|\geq r}\big[\chi_{A_k}(x+p(x,z))-\chi_{A_k}(x)\big] \nu(dz)dx\\
       & = \sum_k \lambda_u( A_k).
    \end{align*}
    In view of Remark \ref{rem B and L measurable},
    $\chi_{A_k}(x+p(x,z))$ is $\nu$-measurable for almost
    every $x$, and integration and summation commute by the dominated
    convergence theorem since $u\in L^1$ and (H3) holds.

    \noindent 3.\quad      By the Radon-Nikodym theorem,
    there is a unique $w_u\in L^1(\Real^d)$ such that $\lambda_u(A)=\int_{A}w_u(x)dx$
    for  measurable $A\subset \R^d$. In other words, for any $u\in
    L^1(\R^d)$, we may identify $J_ru$ with $w_u$. Moreover, by the
    definition of total variation of $\lambda_u$ and Lemma
    \ref{lem J_r u sig-add} again, $\|w_u\|_{L^1}=|\lambda_u|(\R^d)\leq
    2\|u\|_{L^1}\nu(\{|z|>r\}).$ The proof is complete.
\end{proof}

\section{Dissipative operators -- proof of Theorem \ref{prop Ar+Ir and adj dspt}}\label{sec dspt of L}
This whole section is devoted to the proof that the operators $A_r, I_r,
J_r$ and their adjoints are dissipative, i.e. to prove Theorem
\ref{prop Ar+Ir and adj dspt}.

\subsection{Analysis on $I_r$}\label{subsec dspt of Ir}
Consider $u\in D(L)$, and let
\[
    V:=\set{x\in \Real^d: u(x)=0}.
\]
Denote $w:=|u|$, and  decompose
\[
    V^c=\set{ u(x)>0}\cup \set{u(x)<0}=:V^+\cup V^-.
\]

Denote $I_r:=S+T$, where $S$ is the principal non-Local term as
\begin{equation}\label{eq principal small jump op}
  S u(x):=\int_{{|z|<r}} [u(x-q(x,z))-u(x)+q(x,z)D u(x)] m(x,z)\nu (dz).
\end{equation}

\begin{lem}\label{lem Iu vs Iw}
    Assume (H2) and (H3). Then $A_r w$, $Tw$, and $Sw$ are well-defined on $V^c$.
    In addition, there exists a non-negative function $R(x)$ such that
  \begin{equation}\label{eq Sw =Su+R}
    Sw(x)= \begin{cases}
                R(x)+Su(x),    & x\in V^+, \\
                R(x)-Su(x),    & x\in V^-.
            \end{cases}
  \end{equation}
\end{lem}
\begin{proof}
    Obvious $A_r w$ and $Tw$ are well-defined on $V^c$,
    since they are local operators and $V^c$ is an open set
    where $w=\pm u$.

    For $x\in V^c$. Recall $x=y+p(y,z)$, and $m(x,z)\geq 0$. Hence
    for any $|z|<r$, denote
    \begin{align*}
      F_x^+:= & \set{z<|r|: u(x-q(x,z))>0},\\
      F_x^-:= & \set{z<|r|: u(x-q(x,z))<0},\\
      F_x^0:= & \set{z<|r|: u(x-q(x,z))=0}.
    \end{align*}

    If $x\in V^+$, that is, $u(x)>0$, we observe that a neighborhood of $0$
    is contained in $F_x^+$. Then there holds
    \begin{align*}
      Sw(x)  = & \int_{{|z|<r}} [w(x-q(x,z))-w(x)+q(x,z)D w(x)] m(x,z)\nu (dz) \\
      = & \int_{F_x^+} [u(x-q(x,z))-u(x)+q(x,z)D u(x)] m(x,z)\nu (dz)\\
            & +\int_{F_x^-} [-u(x-q(x,z))-u(x)+q(x,z)D u(x)] m(x,z)\nu (dz)\\
            & +\int_{F_x^0} [0 -u(x)+q(x,z)D u(x)] m(x,z)\nu (dz)\\
      = & \int_{F_x^+} [u(x-q(x,z))-u(x)+q(x,z)D u(x)] m(x,z)\nu (dz)\\
            & +\int_{F_x^-} [u(x-q(x,z))-u(x)+q(x,z)D u(x)\\
            & \qquad -2u(x-q(x,z))] m(x,z)\nu (dz)\\
            & +\int_{F_x^0} [0 -u(x)+q(x,z)D u(x)] m(x,z)\nu (dz)\\
      = & Su(x)-2\int_{F_x^-} u(x-q(x,z))m(x,z)\nu (dz).
    \end{align*}
    The last term is then non-negative and point-wisely finite.

    Similarly, if $x\in V^-$,
    \begin{align*}
      Sw(x)  = & \int_{{|z|<r}} [w(x-q(x,z))-w(x)+q(x,z)D w(x)] m(x,z)\nu (dz)  \\
         = & \int_{{|z|<r}} [w(x-q(x,z))+u(x)-q(x,z)D u(x)] m(x,z)\nu (dz) \\
         = & \int_{F_x^+} [u(x-q(x,z))+u(x)-q(x,z)D u(x)] m(x,z)\nu (dz)\\
            & +\int_{F_x^-} [-u(x-q(x,z))+u(x)-q(x,z)D u(x)] m(x,z)\nu (dz)\\
            & +\int_{F_x^0} [0 +u(x)-q(x,z)D u(x)] m(x,z)\nu (dz)\\
         = & -\int_{{|z|<r}} [u(x-q(x,z))-u(x)+q(x,z)D u(x)] m(x,z)\nu (dz)\\
            & + 2\int_{F_x^+} u(x-q(x,z))m(x,z)\nu (dz)\\
         = & -Su(x)+ 2\int_{F_x^+} u(x-q(x,z))m(x,z)\nu (dz).
    \end{align*}

    Therefore we obtained the following relationship
    \begin{align*}
      S w(x)&=\begin{cases}
              S u(x)-2\int_{F_x^-} u(x-q(x,z))m(x,z)\nu (dz),& x\in V^+, \\
              -S u(x)+ 2\int_{F_x^+} u(x-q(x,z))m(x,z)\nu (dz),& x\in V^-.
            \end{cases} \\
       &=: \begin{cases}
              S u(x)+R(x),    & x\in V^+, \\
              -S u(x)+ R(x), & x\in V^-.
            \end{cases}
    \end{align*}
    The proof is complete.
\end{proof}

\begin{lem}\label{lem int Iw leq 0}
    Assume (H2) and (H3). Then $\int_{V^c} I_r w(x)dx \leq 0$.
\end{lem}
\begin{proof}
    By definition, we can write for $x\in V^c$
    \begin{align*}
      I_r w(x) = & \int_{{|z|<r}} \Big([w(x-q(x,z))-w(x)+q(x,z)D w(x)]  m(x,z) \\
         & + (D w(x))^T [p(x,z)-q(x,z) m(x,z) ] \\
         & + w(x)  [ m(x,z) +\text{div}_x p(x,z)-1]\Big)\nu (dz)\\
         =: & \int_{{|z|<r}} h(x,z)\nu (dz).
    \end{align*}

    Next we note that there exist constants $C>0$ and $R>0$ such that
    \begin{equation}\label{eq lower bound for h(x,z)}
      h(x,z)\geq -C|z|^2\chi_{\set{|x|\leq R}}(x)=:-g(x,z).
    \end{equation}
    This is true in view of $u\in D(L)$, Lemma \ref{lem Iu vs Iw},
    and the discussion in Section \ref{subsec well def of I_r}.
    Evidently by the assumptions,
     $0\leq \int_{\Real^d}\int_{|z|<r} g(x,z)\nu(dz)dx<\infty$.
    So $g$ is an integrable lower bound.

    Then we truncate the integrand by defining
    \[
        h_n(x,z):=h(x,z)\chi_{\set{r/n \leq |z|<r}}(z),\qquad n=1,\,2,\,3,\dots
    \]
    Obviously $h_n(x,z) \geq \min\set{h(x,z),0}\geq -g(x,z)$,
    and $\lim_n h_n(x,z)=h(x,z)$,
    for all $(x,z)\in V^c\times \set{|z|<r}$.
    Then we claim that for all $n=1,\,2,\,3,\dots$
    \begin{equation}\label{eq int int h_n <=0}
      \int_{V^c}\int_{|z|<r} h_n(x,z)\nu(dz)dx \leq 0.
    \end{equation}

    With \eqref{eq lower bound for h(x,z)},
    \eqref{eq int int h_n <=0}, and the integrable lower bound $g(x,z)$,
    we can apply Fatou's Lemma and prove Lemma \ref{lem int Iw leq 0} by
    \begin{align*}
      \int_{V^c} I_r w(x)dx  & = \int_{V^c}\int_{|z|<r} h(x,z)\nu(dz)dx
        = \int_{V^c}\int_{|z|<r} \lim_n h_n(x,z)\nu(dz)dx\\
       & \leq \liminf_n \int_{V^c}\int_{|z|<r} h_n(x,z)\nu(dz)dx
         \leq \liminf_n 0 =0.
    \end{align*}

    The rest of the proof will be used to prove Claim \eqref{eq int int h_n <=0}.
    Observe that by definition
    \[
        \int_{V^c}\int_{|z|<r} h_n(x,z)\nu(dz)dx
            =\int_{V^c}\int_{{r/n\leq |z|<r}} h(x,z)\nu(dz)dx,
    \]
    and the L\'evy measure $\nu$ is no longer singular on
    the set $\set{r/n\leq |z|<r}$. Hence
    \begin{align}\label{eq int r/n<z h}
       &\notag \int_{V^c}\int_{r/n\leq |z|<r} h(x,z)\nu(dz)dx \\
       & =\int_{V^c}\int_{r/n\leq |z|<r}
            \big[ w(x-q(x,z))m(x,z)-w(x) +\text{div}_x (w (x)p(x,z)) \big] \nu(dz)dx
    \end{align}

    Then we consider the first two terms in \eqref{eq int r/n<z h}.
    \begin{align*}
       & \int_{V^c}\int_{{r/n\leq |z|<r}}[ w(x-q(x,z))m(x,z)-w(x)]\nu (dz)dx \\
       & =\int_{{r/n\leq |z|<r}}\int_{V^c}[ w(x-q(x,z))m(x,z)-w(x)]dx \nu (dz)\\
       & =\int_{{r/n\leq |z|<r}}\left( \int_{\Real^d}[ w(x-q(x,z))m(x,z)-w(x)]dx \right.\\
       &    \qquad\left. -\int_{V}[ w(x-q(x,z))m(x,z)-w(x)]dx \right)\nu (dz)\\
       & \leq \int_{{r/n\leq |z|<r}}\left( \int_{\Real^d}
            [ w(x-q(x,z))m(x,z)-w(x)]dx\right)\nu (dz),
    \end{align*}
    since $w=|u|=0$  on $V$ and $w\geq 0$.
    For the last term, we observe
    \begin{align*}
        & \int_{{r/n\leq |z|<r}}\left( \int_{\Real^d}
            [ w(x-q(x,z))m(x,z)-w(x)]dx\right)\nu (dz)\\
       & = \int_{{r/n\leq |z|<r}}(\norm{u}_1-\norm{u}_1)\nu (dz)=0.
    \end{align*}

    Now it remains to consider the third term in \eqref{eq int r/n<z h}.
    By Fubini's theorem, we have
    \begin{align*}
       & \int_{V^c}\int_{{r/n\leq |z|<r}} \text{div}_x (w (x)p(x,z))\nu (dz)dx\\
      & =\int_{{r/n\leq |z|<r}}\int_{V^c}\text{div}_x (w (x)p(x,z)) dx \nu (dz).
    \end{align*}

    Since $V^c=\set{u\neq 0}$, we now claim that
    \begin{equation}\label{eq int div =0}
      \int_{V^c}\text{div}_x (w (x)p(x,z)) dx=0.
    \end{equation}
    Then $\int_{V^c}\int_{{r/n\leq |z|<r}} \text{div}_x (w (x)p(x,z))\nu (dz)dx
        =\int_{{r/n\leq |z|<r}} 0 \nu (dz)=0$.

    Finally, we are to show \eqref{eq int div =0}.
    Without loss of generality, we assume the set $\set{u\neq 0}$
    has piece-wise $C^1$ boundary, otherwise
    we can approximate $V^c$ by sets $\set{|u|>\eps_n}$,
    $0<\eps_n\to0$, with $C^1$ boundaries. This can be done since by
    Sard's theorem and the implicit function theorem, $\set{|u|>\eps}$ has
    $C^1$ boundary for a.e. $0<\eps<\max |u|$.
    More details can be found in the proof of
    Proposition 3.4 in \cite{CNJLocal}.

    Because $u=0$ on both $\partial\set{u> 0}$ and $\partial\set{u< 0}$, we have
    \begin{align*}
       & \int_{\set{u\neq 0}}\text{div}_x (w (x)p(x,z)) dx \\
       & =\int_{\set{u> 0}}\text{div}_x (u (x)p(x,z)) dx
            -\int_{\set{u< 0}}\text{div}_x (u (x)p(x,z)) dx\\
       & = \int_{\partial\set{u> 0}} u (x)p(x,z) \vec{n}dS
            -\int_{\partial\set{u< 0}} u (x)p(x,z) \vec{n}dS\\
       & =0-0=0,
    \end{align*}
    where $\vec{n}$ denotes the outer unit normal vector.

    Now the proof of Claim \eqref{eq int int h_n <=0} is complete.
\end{proof}

\subsection{Dissipativity of $A_r+I_r$}\label{subsec dspt of Ar+Ir}
\begin{proof}[Proof of Proposition \ref{prop Ar+Ir and adj dspt} \textbf{(a)}]
    The sum of dissipative operators are in general not necessarily
    dissipative operators. However in our case,
    we are able to show that for any $\lambda>0$,
    there holds the dissipativity inequality
    \[
        \norm{(\lambda-(A_r+I_r))u}_1 \geq \lambda \norm{u}_1.
    \]

    Recall the decomposition $I_r=S+T$ in
    Lemma \ref{lem Iu vs Iw}. We rewrite
   \begin{align*}
     & \norm{(\lambda-(A_r+I_r))u}_1
     := \int_{\Real^d} \abs{\lambda u-A_ru-I_ru}\\
      &=\int_{\Real^d} \abs{\lambda u-A_ru-Tu-Su}\\
     &    \geq \int_{ V^c} \abs{\lambda u-A_ru-Tu-Su}\\
     & =\int_{V^+} \abs{\lambda u-A_ru-Tu-Su)}
       +\int_{V^-} \abs{\lambda u-A_ru-Tu-Su}.
   \end{align*}

   Then we use the relationship \eqref{eq Sw =Su+R}
   and yield
   \begin{align*}
      & \int_{V^+} \abs{\lambda u(x)-A_ru(x)-Tu(x)-Su(x)} dx \\
        & \quad +\int_{V^-} \abs{\lambda u(x)-A_ru(x)-Tu(x)-Su(x)} dx \\
      & = \int_{V^+} \abs{\lambda w(x)-A_rw(x)-Tw(x)-(Sw(x)-R(x))} dx  \\
        & \quad +\int_{V^-} \abs{\lambda (-w)(x)-A_r(-w)(x)-T(-w)(x)-(R(x)-Sw(x))} dx \\
      & = \int_{V^+} \abs{\lambda w(x)-A_rw(x)-Tw(x)-Sw(x)+R(x)}  dx \\
        & \quad +\int_{V^-} \abs{\lambda w(x)-A_rw(x)-Tw(x)-Sw(x)+R(x)} dx .
   \end{align*}

    So we got the same integrand on $V^+$ and $V^-$, that is
   \begin{align*}
        & \int_{V^+} \abs{\lambda w(x)-A_rw(x)-Tw(x)-Sw(x)+R(x))}  dx \\
        & \quad +\int_{V^-} \abs{\lambda w(x)-A_rw(x)-Tw(x)-Sw(x)+R(x))}  dx \\
        & = \int_{V^c} \abs{\lambda w(x)-A_rw(x)-Tw(x)-Sw(x)+R(x)} dx \\
      & = \int_{V^c} \abs{\lambda w(x)-A_rw(x)-I_rw(x)+R(x)} dx .
   \end{align*}

   Keep in mind that $R(x)\geq 0$ for all $x\in V^c$,
   and we can estimate that
   \begin{align*}
      & \int_{V^c} \abs{\lambda w(x)-A_rw(x)-I_rw(x)+R(x)} dx \\
      & \geq \int_{V^c} (\lambda w(x)-A_rw(x)-I_rw(x)+R(x)) dx \\
      & \geq \int_{V^c} (\lambda w(x)-A_rw(x)-I_rw(x)) dx \\
      & =\int_{\Real^d} \lambda w(x)-\int_{V^c} (A_rw(x)+I_r w(x)) dx \\
      & \geq \lambda \norm{u}_1,
   \end{align*}
   since $\int_{V^c} A_rw\leq 0$ from the proof of Proposition 3.4 in \cite{CNJLocal},
   and $\int_{V^c} I_r w \leq 0$ by Lemma \ref{lem int Iw leq 0}.
\end{proof}

\subsection{Dissipativity of $A_r^*+I_r^*$}
\label{subsec dspt of Ar*+Ir*}
We specify the domain of the adjoint operator,
\begin{align*}
  &  D(A_r^*+I_r^*):=\\
  & \set{f\in L^{\infty}(\Real^d): \exists\, g \in L^{\infty}(\Real^d)
        \text{ such that } \forall u\in D(L), \int_{\Rd} g u=\int_{\Rd} f (A_r+I_r) u}.
\end{align*}
Then  $(A_r^*+I_r^*) f=g$ in the distributional sense.

\begin{proof}[Proof of Proposition \ref{prop Ar+Ir and adj dspt} \textbf{(b)}]
    Consider an arbitrary function $f\in D(A_r^*+I_r^*)$,
    it then follows from  Condition (E) that
    $f\in W^{2,p}_{loc}(\Rd)$ for some $p>d$.
    Hence we construct $f_n(y):= f(y)-\frac{|y|^2}{n}=:f(y)-g_n(y)$.
    Then by Lemma 5.9 in \cite{CNJLocal},
    there exists a sequence $y_n'\in \Real^d$ such that
    $y_n'$ is the global maximal point of $f_n$ with
    \begin{equation}\label{eq property of yn}
        \lim_n \abs{f(y_n')} =\norm{f}_{\infty},\,
        \lim_n (1+|y_n'|)|D f(y_n')|  = 0, \text{ and }
        \lim_n \frac{|y_n'|^2}{n}=0.
    \end{equation}

    Without loss of generality, we require $f(y_n')\geq 0$.
    Then in view of assumptions (H1) -- (H3) and
    Lemma \ref{lem bound of m(x,z)}, for each $n$ we are able to apply a version of the
    Bony maximum principle -- Proposition 3.1.14
    in \cite{garroni2002second} --  in any bounded neighborhood of $y_n'$
    and obtain
    \begin{equation}\label{eq lim I+A fn}
        \lim_{\rho \to 0} \essinf_{B(\rho ,y_n')} (-A_r^*-I_r^*)f_n(y_n')\geq 0.
    \end{equation}
    (Note that Proposition 3.1.14 holds without uniform ellipticity as can easily
    be seen from its proof given in \cite{GimbertLions1984}).
    We also avoid the points where $D^2 f$ is not defined
    and pick, for each $n$, another point $y_n$
    such that
    \begin{equation}\label{eq compare yn and yn'}
        \abs{y_n-y_n'}+\abs{f(y_n)-f(y_n')}
            +\abs{D f(y_n)}\leq \frac 1n,\quad
            (-A_r^*-I_r^*)f_n(y_n)\geq -\frac 1n.
    \end{equation}
    And we can always take
    \[
        \set{y_n}\subset \set{y\in\Real^d: |\lambda f(y)- A_r^*f(y)-I_r^*f(y)|
            \leq \norm{\lambda f-(A_r^*+I_r^*)f}_{\infty}},
    \]
    because the complement of the latter set has zero Lebesgue measure in $\Real^d$.
    Hence
    \begin{align*}
      &(-A_r^*-I_r^*)f(y_n)  \\
      & = (-A_r^*-I_r^*)f_n(y_n)+(-A_r^*-I_r^*)g_n(y_n)\\
      & \geq  -\frac 1n-\abs{(A_r^*+I_r^* )g_n(y_n)} \\
       & =  -\frac 1n - \frac 1n \abs{2 \langle y_n,b(y_n)\rangle
            +\sigma^T(y_n)\sigma(y_n)+\int_{|z|<r}p^T(y_n,z)p(y_n,z)
                \nu(dz) }.
    \end{align*}
    By \eqref{eq property of yn} and
    \eqref{eq compare yn and yn'}, (H1) -- (H3),
     the right hand side of above tends to zero. Therefore
    $$\liminf_n (-A_r^*-I_r^*)f(y_n)\geq 0.$$

    Finally for all $\lambda>0$, there holds
    \begin{align*}
      \lambda\norm{f}_{\infty} & =\lambda\lim_n f(y_n)
        =\liminf_n  \lambda f(y_n) \\
       & \leq \liminf_n  \lambda f(y_n)
        +\liminf_n(-A_r^*-I_r^*)f(y_n)\\
       & \leq \liminf_n \big( \lambda f(y_n) -(A_r^*+I_r^*)f(y_n) \big)\\
       & \leq \norm{ (\lambda-A_r^*-I_r^*)f}_{\infty}.
    \end{align*}
    The proof is complete.
\end{proof}
\begin{rem}
    \textbf{(a)} Maximum principles like \eqref{eq lim I+A fn}
    first appeared in \cite{Bony1967} for local operators with
    $p>d$ and the critical case $p=d$ was treated in \cite{Lions1983},
    and the first treatment of non-local operators is found in
    \cite{GimbertLions1984} with $p(y,z)=z$, and the proof can
    be easily extended to functions $p(y,z)$ locally bounded in
    $y$, see Section 3.1 in \cite{garroni2002second}.

    \textbf{(b)} The infimum in \eqref{eq lim I+A fn} cannot be replaced by supremum.
    Indeed, let $f(x):=x^2\big(\sin(\ln(x^2))-2\big)\in W^{2,\infty}_{loc}(\R)$,
    with global maxima $f(0)=0$.
    Then $f''(x)=6\cos(\ln(x^2))-2\sin(\ln(x^2))-4$ and
    $\lim_{r\to 0} \essinf_{B(r,0)} f''=-10$ but on the other hand
    $\lim_{r\to 0} \esssup_{B(r,0)} f''=2$.
\end{rem}

\subsection{Dissipativity of $J_r$}\label{subsec dspt of Jr}
Proposition \ref{prop Ir and Jr well def} \textbf{(b)} claims that
$J_r$ is well defined on $L^1(\Rd)$,
but in general it does not possess an explicit
expression since \eqref{eq global est of 1+p'x}
may no longer hold.

In view of Section 1.1.4 in \cite{pazy1992semigroups},
for each $u\in L^1(\Real^d)$, we define its \textit{duality set}
\[
    \mathcal J(u):=\set{f\in L^{\infty}(\Real^d):
        \langle u, f\rangle=\norm{u}_1^2=\norm{f}_{\infty}^2 }.
\]
An equivalent definition for dissipativity is that
for all $u\in L^1(\Real^d)$ there exists $f\in \mathcal J(u)$
such that $ \langle J_r u, f\rangle\leq 0$,
where $\langle \cdot, \cdot\rangle$ stands for duality pairing,
cf. Definition 1.4.1 and Theorem 1.4.2 in \cite{pazy1992semigroups}.

\begin{proof}[Proof of Theorem \ref{prop Ar+Ir and adj dspt} \textbf{(c)}]
    For any $u\in L^1(\Real^d)$, take
    $f_u(x):=\norm{u}_1 \text{sign }u(x)$.

    Obviously $f_u\in\mathcal J(u)$. Moreover
    \begin{align*}
      \langle J_r u, f_u\rangle & =\langle u, J_r^* f_u\rangle
         =\int_{\Rd} u J_r^* f_u\\
       & =\int_{\Rd}  u(x)\norm{u}_1 \int_{|z|\geq r} [\text{sign }u(x+p(x,z))-\text{sign }u(x)]\nu (dz)dx\\
       & =\norm{u}_1\int_{\Rd}  \int_{|z|\geq r}[ u(x)\,\text{sign }u(x+p(x,z))-|u|(x)]\nu (dz)dx\\
       & \leq 0,
    \end{align*}
    since the integrand is always non-positive.
\end{proof}

\section{Elliptic Regularity -- proof of Proposition \ref{prop ell reg 01}}\label{sec ell reg}
The proof of part (a) is similar to the proof of part (b), but easier
since the non-local operator can be treated
as a lower order perturbation. In this case the proof
follows from arguing as in the proof of
Theorem 1.5 in \cite{ZhangBao2013} and applying Proposition 1.1
in \cite{ZhangBao2013} and Lemma 2.3.6 in \cite{garroni2002second}.

By condition (HE) with $s=1$, we know $\int_{|z|<1}|z|\nu(dz)<\infty$. Hence
for $C^1$ functions the following operators are well-defined.
\begin{align*}
  I_r^* \phi(x) & = \tilde I_r^* \phi(x) + D\phi(x)\int_{|z|<r} p(x,z)\nu(dz),\\
  \tilde I_r^* \phi(x)& := \int_{|z|<r}\big( \phi(x+p(x,z))-u(x)\big)\nu(dz),\\
  \tilde I_r \phi(x) & := (\tilde I_r^* )^* \phi(x)
    =\int_{|z|<r}\big( \phi(x-q(x,z))-\phi(x)\big)m(x,z)\nu(dz)\\
        &\qquad\qquad \qquad\qquad   +\phi(x)\int_{|z|<r}\big( m(x,z)-1\big)\nu(dz),
\end{align*}
where $m$ and $q$ are defined in the beginning of Section \ref{sec well-def of L}.

\begin{proof}[Proof of Proposition \ref{prop ell reg 01} (a)]
    Fix $R>0$ and let $B_R:=\set{x\in\Rd: |x|\leq R}$.
    By assumptions and (HE) with $s=1$, equation \eqref{eq distr sol}
    is equivalent to
    \begin{equation}\label{eq div}
      \int_{\Rd}  \sum_{i,j=1}^d f \partial_i \left(\frac 12 a_{ij}\partial_j u\right)
         =\int_{\Rd}  [gu + fH-f \tilde I_r u]
    \end{equation}
    where
    \begin{align*}
      H(x) := & \sum_{i=1}^d \partial_i(b_i(x) u(x))
            -\frac 12 \sum_{i,j=1}^d \partial_i (\partial_j a_{ij}(x) u(x))\\
            &    +\int_{|z|<1}\text{div }\big(u(x)\, p(x,z)\big)\nu (dz).
    \end{align*}
    With $r<1/(4dK)$ as in Section \ref{subsec well def of I_r},
    one can apply Lemmas \ref{lem est of q(x,z)} and \ref{lem bound 1-m(x,z)}
    and verify all assumptions of Lemma 2.3.6 in \cite{garroni2002second}.
    Hence for all $1\leq q \leq \infty$,
    \begin{equation}\label{eq Lq est of Ir}
      \norm{\tilde I_r u}_{L^q(B_{R})}\leq C(R) \norm{u}_{W^{1,q}(B_{R})}.
    \end{equation}

    Next let $q:=\frac{p}{p-1}$ and we argue as in
    the proof of Theorem 1.5 in \cite{ZhangBao2013}
    to show that $f\in \bigcap_{1<p<\infty}W^{1,p}_{loc}(\Rd)$.

    Let $|h|< R/6$, $k=1,\cdots,d$, and $\eta\in C_c^\infty(\Rd)$ be
    such that $\supp \eta\subset B_{2R/3}$ and $\eta\equiv 1$ in $B_{R/2}$.
    Since $f\in L^\infty$, the difference quotient
    \[
        \Delta_h^k f(x):=\frac{f(x+he_k)-f(x)}{h}\in L^p(B_R)
    \]
    and hence
    \[
        \abs{\eta\Delta_h^k f}^{p-1}\sign (\Delta_h^k f)\in L^q(B_R).
    \]
    By the classical results (see Theorem 9.15 in \cite{gilbarg2015elliptic}),
    there is a unique strong solution in $W_0^{1,q}(B_R)\cap W^{2,q}(B_R)$
    for the following Dirichlet problem
    \[
        \frac 12 a_{ij}\partial_i\partial_j v_h
            =\abs{\eta\Delta_h^k f}^{p-1}\sign (\Delta_h^k f)
            \quad \text{ in } B_R,\qquad v_h=0 \quad\text{ in }\partial B_R,
    \]
    such that
    \begin{equation}\label{eq est on vh}
        \norm{v_h}_{W^{2,q}(B_R)}
            \leq C\norm{\abs{\eta\Delta_h^k f}^{p-1}}_{L^q(B_R)}
            \leq C\norm{\eta\Delta_h^k f}^{p-1}_{L^p(B_R)}.
    \end{equation}
    By a density argument, \eqref{eq div} holds for
    all elements $u\in W^{2,q}(\Rd)$ with compact support.
    Hence we choose $u:=\eta \Delta_{-h}^k v_h$ and
    follow the arguments in \cite{ZhangBao2013} for the local terms to get
    \begin{align*}
      \norm{\eta \Delta_h^k f}_{L^p(B_R)}^p
       \leq & \frac 14 \norm{\eta \Delta_h^k f}_{L^p(B_R)}^p +C\\
       & +\norm{f\int_{|z|<1}\text{div}_x (p(\cdot,z)u)\nu(dz)}_{L^1(B_R)}
            +\norm{f\tilde I_r u}_{L^1(B_R)}.
    \end{align*}

    By H\"older's and Young's inequalities, \eqref{eq est on vh},
     (H2), and $\int_{|z|<1}|z|\nu(dz)<\infty$,we have
    \begin{align*}
       &\int_{B_R}\abs{f(x)\int_{|z|<1}\text{div }\big(p(x,z)u(x)\big)\nu (dz)}dx\\
       &  \leq C\int_{|z|<1}|z|\nu(dz) \norm{f}_{L^p(B_R)}\norm{v_h}_{W^{1,q}(B_R)}\\
       & \leq C \norm{f}_{L^p(B_R)}\norm{\eta\Delta_h^k f}^{p-1}_{L^p(B_R)}\\
       &  \leq \frac 14 \norm{\eta\Delta_h^k f}^p_{L^p(B_R)}+C.
    \end{align*}
    Similarly, with the aid of \eqref{eq Lq est of Ir},
    \begin{align*}
       &\int_{B_R}\abs{f(x)\tilde  I_r u(x)} dx
         \leq \norm{f}_{L^p(B_R)}\norm{v_h}_{W^{1,q}(B_R))}\\
       & \leq C \norm{f}_{L^p(B_R)}\norm{\eta\Delta_h^k f}^{p-1}_{L^p(B_R)}
          \leq \frac 14 \norm{\eta\Delta_h^k f}^p_{L^p(B_R)}+C.
    \end{align*}

    Combining the above estimates we get that
    $\norm{\eta \Delta_h^k v_h}_{L^p(B_R)}^p\leq 4C$. Thus
    by definitions of $\eta$ and $v_h$,
    \[
        \norm{\Delta_h^k f}_{L^p(B_{R/2})}\leq 4C,
            \text{ for all }|h|\in (0, 1/6).
    \]
    By the property of difference quotients,  (cf. e.g. Theorem 5.8.3 in
    \cite{evans2010partial}), we have that the weak derivative
    $\partial_k f$ exists in $B_{R/2}$ and
    $\norm{\partial_k f}_{L^p(B_{R/2})}<\infty$.
    Since $R>0$ is arbitrary, $f\in \bigcap_{1<p<\infty}W^{1,p}_{loc}(\Rd)$.

    With this extra regularity we can further rearrange \eqref{eq div} as
    \[
        \int_{\Rd}  f(x)  \sum_{i,j=1}^d \partial_i
            \left(\frac 12 a_{ij}(x)\partial_j u(x)\right)dx
            =\int_{\Rd}  u\,[g + G-\tilde I_r^*f]
    \]
    where
    \[
        G(x) := \sum_{i=1}^d \partial_i f(x)
            \left(\frac 12 \sum_{j=1}^d \partial_j a_{ij} -b_i(x)\right)
           +\int_{|z|<1} Df(x)p(x,z) \nu(dz).
    \]
    By the regularity of $f$, (H2) and (HE) with $s=1$, and
    Lemma 2.3.6 in \cite{garroni2002second}
    again we observe that $g$, $G$, and
    $\tilde I_r^*f \in L^p_{loc}(\Rd)$ for all $1<p<\infty$.
    Hence by Proposition 1.1 in \cite{ZhangBao2013} we have
     $f\in \bigcap_{1<p<\infty}W^{2,p}_{loc}(\Rd)$.
\end{proof}

Now we will prove part (b), and we start with the following estimate.
\begin{lem}\label{lem Lip q}
    Assume (H2) with constant $K>0$. Then
    \[
        \abs{q(x_1,z)-q(x_2,z)}\leq 2K|z||x_1-x_2|,\quad
            \text{for all } x_1,\,x_2\in \Rd \text{ and } |z|<1/2K.
    \]
\end{lem}
\begin{proof}
    By definition
    \begin{align*}
      & \abs{q(x_1,z)-q(x_2,z)}  =\abs{p(y(x_1,z),z)-p(y(x_2,z),z)} \\
       & = \abs{p(x_1-q(x_1,z),z)-p(x_2-q(x_2,z),z)}\\
       & \leq K |z|\abs{x_1-x_2+ q(x_2,z)-q(x_1,z)}\\
       & \leq K |z|\abs{x_1-x_2}+\frac 12 \abs{ q(x_2,z)-q(x_1,z)}.
    \end{align*}
    Hence $\abs{q(x_1,z)-q(x_2,z)}\leq 2K|z||x_1-x_2|$.
\end{proof}

\begin{proof}[Proof of Proposition \ref{prop ell reg 01} (b)]
    The first two steps below mainly follow the same ideas of proving
    Theorem 1.5 in \cite{ZhangBao2013} and Theorem 1.3 in \cite{Zhang20121867}
    respectively.

    \emph{Step I. } $W^{1,p}_{loc}$ regularity for $p\in (1, \infty)$.

    Let $R>1$ and $B_R(0)$ be the ball of radius $R$ centered at the origin.
    For any $x_0\in B_R(0)$, denote $B_r$ the ball of radius $r>0$ centered at $x_0$.

    According to (H2) and Lemma \ref{lem est of q(x,z)},
    for this $R>0$ there exists $r_R>0$ such that
    \[
        \sup_{|x|<R,\,|z|<r_R} |p(x,z)|\leq 1/6\quad\text{ and }\quad
        \sup_{|x|<R,\,|z|<r_R} |q(x,z)|\leq 1/6.
    \]
    This is  the principal part of the operator $I_r$.
    The outer part of $I_r$,
    \[
        \int_{r_R\leq |z|<r} [f(x+p(x,z))-f(x)-p(x,z)D f(x)]\nu(dz),
    \]
    is a bounded operator in $L^p(\Rd)$ (cf. Lemma 2.3.6 in \cite{garroni2002second})
    plus a first-order differential operator.
    So they can be absorbed into other lower order terms.
    Hence without loss of generality, one can suppose $r_R=r$.

    Let $p>d$ be so large that the conjugate index $q:=\frac{p}{p-1}<\frac{d}{s-1}$.
    Let $\eta\in C_c^\infty(\Rd)$ be a
    nonnegative truncation function
    such that $\supp \eta\subset B_{2/3}$ and $\eta\equiv 1$ in $B_{1/2}$.
    Then for $|h|\in (0,1/6)$ and $k=1,\dots,d$,  difference quotient
    \[
        \Delta_h^k f(x):=\frac{f(x+he_k)-f(x)}{h}\in L^p(B_{1}),
    \]
    and hence
    \[
        \abs{\eta\Delta_h^k f}^{p-1}\sign (\Delta_h^k f)\in L^q(B_1).
    \]

    Since $\int_{|z|<1}|z|^s \nu(dz)<\infty$,
    we can take $\gamma_1=s$ to be the \emph{boundary order} of the operator $I_r^*$
    defined in Theorem 3.1.22 in \cite{garroni2002second}.
    By this theorem there exists a
    unique strong solution of the Dirichlet problem
    $v_h\in W_0^{1,q}(B_1)\cap W^{2,q}(B_1)$:
    \begin{equation}\label{eq Dirichlet vh}
        (A_r+I_r) v_h
         =\abs{\eta\Delta_h^k f}^{p-1}\sign (\Delta_h^k f)\qquad\text{ in }B_1,
            \qquad v_h=0 \qquad\text{ in }B_1^c,
    \end{equation}
    satisfying the estimate
    \begin{equation}\label{eq est of vh}
        \norm{v_h}_{W^{2,q}(B_1)}\leq C\norm{\abs{\eta\Delta_h^k f}^{p-1}}_{L^q(B_1)}
            = C\norm{\eta\Delta_h^k f}_{L^p (B_1)}^{p-1}.
    \end{equation}
    Next we denote
    \[
      I_r (\eta \Delta_{-h}^k v_h)(x)
       = \Delta_{-h}^k \big( \eta(x) I_r v_h(x)\big)
       +R(x,h).
    \]
    The commutator $R(x,h)$ will be explicitly computed in
    \emph{Step III} where we will also establish the following
    estimate
    \begin{equation}\label{eq est of R}
      \norm{R(\cdot,h)}_{L^q (B_1)}\leq C\norm{v_h}_{W^{2,q}(B_1)},
    \end{equation}
    where $C>0$ is a constant independent of $h$.

    We now continue as in the proof of
    Theorem 1.5 in \cite{ZhangBao2013}. First, observe that
    \begin{equation}\label{eq commutators}
        (A_r+I_r)(\eta \Delta_{-h}^k v_h)
        = \Delta_{-h}^k \big( \eta (A_r+I_r)v_h \big)
            + \tilde R(x,h),
    \end{equation}
    where the commutator $\tilde R(x,h)$,
    consisting of $R(x,h)$ and commutator of the local terms,
    belongs to $L^q_{loc}(\Rd)$ satisfying (see \cite{ZhangBao2013})
    \[
        \|\tilde R(x,h)\|_{L^q(B_1)}
        \leq C\norm{\eta\Delta_h^k f}_{L^p (B_1)}^{p-1}.
    \]
    Since $f\in L^p_{loc}$, it follows from
    a density arguments that equation \eqref{eq distr sol} holds
    for all functions in $W^{2,q}(\Rd)$ with compact support.
    Hence we take $u:=\eta v_h$ in \eqref{eq distr sol}, use \eqref{eq commutators}
    and \eqref{eq Dirichlet vh} consecutively to obtain
    \begin{align*}
       & \int_{B_1} g\eta \Delta_{-h}^k v_h
        = \int_{B_1} f (A_r+I_r)(\eta \Delta_{-h}^k v_h)\\
       & = \int_{B_1} \big[(\Delta_h^k f)  \eta (A_r+I_r)v_h  + f\tilde R\big]
        = \int_{B_1} \big[ \abs{\eta \Delta_h^k f}^p +f\tilde R\big].
    \end{align*}
    By H\"older's inequality, \eqref{eq est of R}, \eqref{eq est of vh},
    and Young's inequality,
    we have further that
    \begin{align*}
      \norm{\eta \Delta_h^k f}_{L^p (B_1)}^p
        & \leq \int_{B_1} \big(|f\tilde R|+|g\eta\Delta_{-h}^k v_h|\big) \\
       & \leq \norm{f}_{L^p(B_1)}\|\tilde R\|_{L^q(B_1)}
        +\norm{g}_{L^p(B_1)}\norm{\eta\Delta_{-h}^k v_h}_{L^q(B_1)}\\
       & \leq C(\norm{f}_{L^p(B_1)}+\norm{g}_{L^p(B_1)})\norm{v_h}_{W^{2,q}(B_1)}\\
       & \leq C'(\norm{f}_{L^p(B_1)}+\norm{g}_{L^p(B_1)})\norm{\eta\Delta_h^k f}_{L^p (B_1)}^{p-1}\\
       & \leq \frac 12 \norm{\eta \Delta_h^k f}_{L^p (B_1)}^p + C''.
    \end{align*}
    Therefore $\norm{\eta \Delta_h^k f}_{L^p (B_1)}^p\leq 2C''$,
    uniformly for $h\to 0$.
    By the property of difference quotient (cf. e.g. Theorem 5.8.3 in
    \cite{evans2010partial}),
    $\partial_k f$ exists and belongs to $L^p (B_{1/2})$.
    Since $k$ and $x_0\in B_R(0)$ were arbitrarily chosen,
    we can use the finite covering theorem
    to conclude that $f\in W^{1,p}(B_R(0))$.

    \emph{Step II. } $W^{2,p}_{loc}$ regularity for $p\in \big(1, \frac{d}{s-1}\big)$.

    By the previous step, $f$ is now a weak solution of equation \eqref{eq distr sol}.
    For a rigorous interpretation of the non-local operator in the weak sense,
    one can refer to Section 3.2 in \cite{garroni2002second}.
    Here we formally rewrite the equation as
    \begin{equation}\label{eq global in div form}
      \frac12 \sum_{i,j=1}^d\partial_i(a_{ij}\partial_j f) + f + I_r^* f =g+R',
    \end{equation}
    where $R'$ is a local term belonging to $L^p_{loc}(\Rd)$
    for all $p\in (1, \infty)$.
    Especially, multiplying
    on both sides of \eqref{eq global in div form}
    the truncation function $\eta$ defined in the previous step, we observe
    $v:=\eta f \in W^{1,p}_0(B_1)$ is a weak solution of the following equation
    \begin{equation}\label{eq in div form}
        \frac12 \sum_{i,j=1}^d\partial_i(a_{ij}\partial_j v) + v + I_r^* v =\eta g+ R'',
    \end{equation}
    where $R''$ consists of local terms,  local commutators which are
    in $L^p(B_1)$ for all $p\in (1, \infty)$,
    and a non-local commutator
    \begin{align*}
       & \int_{|z|<r}\Big[ \Big( \eta(x+p(x,z))-\eta(x)-p(x,z)D\eta(x)\Big)
            f(x+p(x,z))\\
       & \quad +p(x,z)D\eta(x)\Big( f(x+p(x,z))-f(x) \Big) \Big]\nu(dz)
    \end{align*}
    which also belongs to $L^p(B_1)$ for all $p\in (1, \infty)$.
    This commutator is well-defined for $f\in W^{1,p}_{loc}$
    and it is derived by considering a truncation of the L\'evy measure
    by $0<\eps\leq |z|<r$ and passing $\eps$ to $0+$,
    where the resulting truncated operators are well-defined
    for functions in $W^{1,p}(\Rd)$ with compact support.

    By Theorem 3.2.3 in \cite{garroni2002second},
    the \emph{weak maximum principle} holds
    for the operator
    $\frac12 \sum_{i,j=1}^d\partial_i(a_{ij}\partial_j \cdot) + id + I_r^*$
    in  \eqref{eq in div form}.
    Moreover $w\equiv 1$ is a \emph{weak subsolution}
    (cf. (3.2.24) in Section 3.2 of \cite{garroni2002second})
    for the same operator.
    Hence Theorem 3.2.5 in \cite{garroni2002second} can be applied to
    \eqref{eq in div form} and it guarantees the equation \eqref{eq in div form}
    has a unique weak solution
    in $W^{1,p}_0(B_{1})$ for $p\in (1, d/(s-1))$.
    Moreover, according to Theorem 3.1.22 in \cite{garroni2002second},
    equation  \eqref{eq in div form} also admits a strong solution in
    $W_0^{1,p}(B_1)\cap W^{2,p}(B_1)$ for $p\in (1, d/(s-1))$.
    Therefore $v$ is a strong solution and $f\in W^{2,p}(B_{1/2})$.
    With the finite covering theorem, we can show
    $f\in W^{2,p}(B_R(0))$.

    \emph{Step III}. Finally, we prove \eqref{eq est of R}.
    For simplicity, we will write $he_k=h$ when there is no ambiguity.
    After a long but elementary computation we write
    $R(x,h)=\frac{1}{-h}\int_{|z|<r}\sum_{i=1}^8 R_i (x,z,h)\nu(dz)$ with
    \begin{align*}
      & \sum_{i=1}^8 R_i (x,z,h)\\
        &:= \big(\eta(x)-\eta(x-h)\big)\\
            & \quad \cdot\Big[ v_h(x-h-q(x,z))m(x,z)-v_h(x-h)
        + D_x[p(x,z)v_h(x-h)] \Big]\\[0.2cm]
        & +\Big[ \eta(x-q(x,z))m(x,z)-\eta(x)+  D_x[p(x,z)\eta(x)]\Big]\\
            & \quad\cdot\big(v_h(x-q(x,z)-h)-v_h(x-q(x,z))\big)\\[0.2cm]
        & +p(x)D\eta(x)
            \Big[ v_h(x-h)-v_h(x)-v_h(x-q(x,z)-h)+v_h(x-q(x,z)) \Big]\\[0.2cm]
        &   +\eta(x)\Big[ v_h(x-q(x,z))-v_h(x-q(x,z)-h) \Big]
            (m(x,z)+D_x\cdot p(x,z)-1 )\\[0.2cm]
        &   +\eta(x-h)\Big[ v_h(x-q(x-h,z)-h)-v_h(x-h) +q(x-h,z)
            D v_h(x-h)\Big]\\
            & \quad\cdot\big(m(x,z)-m(x-h,z)\big)\\[0.2cm]
        &   +\eta(x-h)D v_h(x) \\
            & \quad\cdot
            \Big[ p(x,z)-q(x,z)m(x,z)-(p(x-h,z)-q(x-h,z)m(x-h,z)) \Big]\\[0.2cm]
        &   +\eta(x-h) v_h(x)\Big[ m(x,z)+D_x\cdot p(x,z)-1 -
            (m(x-h,z)+D_x\cdot p(x-h,z)-1)\Big]\\[0.2cm]
        &   + \eta(x-h)\Big[ v_h(x-q(x,z)-h)-v_h(x-q(x-h,z)-h)\\
            & \quad +q(x,z)D v_h(x-h) - q(x-h,z)D v_h(x-h) \Big] m(x,z).
    \end{align*}
    Recall that $\supp \eta\subset B_{2/3}$,
    $|h|\in (0,1/6)$, $\sup_{|x|<R,\,|z|<r} |p(x,z)|\leq 1/6$
    and $\sup_{|x|<R,\,|z|<r} |q(x,z)|\leq 1/6$.
    In particular $x+p(x,z)-h\in B_1$ for $x\in B_{2/3}$.

    We will show that for $i=1,\cdots,8$,
    \begin{equation}\label{eq est of Ri}
      \norm{\frac{1}{-h}\int_{|z|<r}R_i(\cdot,z,h)\nu(dz)}_{L^q(B_1)}
            \leq C \norm{v_h}_{W^{2,q}(B_1)}.
    \end{equation}

    For $R_1$, $\abs{\eta(x)-\eta(x-h)}\leq C|h|$ and observe
    \begin{align*}
       & v_h(x-h-q(x,z))m(x,z)-v_h(x-h)  + D_x[p(x,z)v_h(x-h)] \\
       & = \Big(v_h(x-h-q(x,z)) -v_h(x-h) +q(x,z)D v_h(x-h) \Big)m(x,z)\\
       & \quad +D v_h(x-h) \Big( p(x,z) -q(x,z)m(x,z) \Big)\\
       & \quad +v_h(x-h) \Big( m(x,z) + D_x p(x,z)-1 \Big).
    \end{align*}
    The first term equals
    \[
        \int_0^1 (1-\theta)q^T(x,z)(D^2 v_h) (x-h-\theta q(x,z))q(x,z) d\theta \,m(x,z).
    \]
    Then the rest of arguments for $R_1$ follows from
    \eqref{eq est of vh} and the proof of
    Proposition \ref{prop Ir and Jr well def} \textbf{(a)}.

    The analysis of $R_2$ is the same as for $R_1$,
    with the roles of $\eta$ and $v_h$ exchanged.

    For $R_3$, estimate \eqref{eq est of Ri} follows from the observation
    \[
        R_3(x,z,h)=h\big(p(x,z)D \eta(x)\big)^T
            \int_0^1\int_0^1  (D^2 v_h)(x-\theta q(x,z)-\xi h) q(x,z)d\xi d\theta.
    \]

    When $i=4$, we obtain \eqref{eq est of Ri}  from the proof of
    Proposition \ref{prop Ir and Jr well def} \textbf{(a)} and the  fact
    \[
        v_h(x-q(x,z))-v_h(x-q(x,z)-h)
            =\int_0^1 he_k (D v_h)(x-q(x,z)-\theta h)d\theta.
    \]

    For $R_5$, it suffices to show
    $\abs{m(x,z)-m(x-h,z)}\leq C_R|z||h|$ for all $|x|\leq R$.
    Indeed, with the same notations as in Section \ref{sec well-def of L},
    denote $M(x,z):= (D_y p)(y(x,z),z)$. Then
    \begin{align*}
      & m(x,z)-m(x-h,z) \\
      & =\frac{1}{1+\tr M(x,z) + P(x,z)}
        -\frac{1}{1+\tr M(x-h,z) + P(x-h,z)} \\
       & = \frac{\tr M(x-h,z)-\tr M(x,z)+P(x-h,z)-P(x,z)}{(1+\tr M(x,z) + P(x,z))(1+\tr M(x-h,z) + P(x-h,z))}
    \end{align*}
    By the global Lipschitz condition on $p(x,z)$ for $x$, the denominator of the last fraction
    is bounded away from $0$ uniformly for all $x\in \Rd$ and all $|z|<r$.
    Moreover by (H2')
    \begin{align*}
       & \abs{\tr M(x-h,z)-\tr M(x,z)} \\
       & \leq \sum_{k=1}^d \abs{(\partial_{y_k} p)(y(x-h,z),z)
                -(\partial_{y_k} p)(y(x,z),z)}\\
        & =\sum_{k=1}^d \abs{(\partial_{y_k} p)(x-h-q(x-h,z),z)
                -(\partial_{y_k} p)(x-q(x,z),z)}\\
         &   \leq C|z||h|.
    \end{align*}
    Similarly, we have $\abs{P(x-h,z)-P(x,z)}\leq C|z|^2 |h|$,
    since $P(x,z)$ are sum of products of at least two $x$-derivatives
    of $p(x,z)$ by the definition of the Jacobian.

    Now we analyze $R_6$. First we compute
    \begin{align*}
      & p(x,z)-q(x,z)m(x,z)-(p(x-h,z)-q(x-h,z)m(x-h,z))\\
      & = p(x,z)-q(x,z)+q(x,z)(1-m(x,z))\\
        &\quad -p(x-h,z)-q(x-h,z)+q(x-h,z)(1-m(x-h,z)) \\
      & = \int_0^1\int_0^1 D_x^2 p(x-\theta q(x,z)+\xi (h+\theta(q(x,z)-q(x-h,z)) ))\\
          &\qquad  \cdot(h+\theta(q(x,z)-q(x-h,z)))q(x,z)d\xi d\theta\\
      &\quad +\int_0^1 (q(x,z)-q(x-h,z)) D_x p(x-h+\theta q(x-h,z)) d\theta\\
      & \quad + (q(x,z)-q(x-h,z))(1-m(x,z))\\
      & \quad + q(x-h,z)(m(x,z)-m(x-h,z))
    \end{align*}
    Then according to  $\abs{D_x^2 p(x,z)}\leq C_R|z|$ from (H2'),
    Lemmas \ref{lem Lip q}, \ref{lem est of q(x,z)} and \ref{lem bound 1-m(x,z)},
    we know the first three terms are bounded by $C_R |z|^2 |h|$ for all $|x|\leq R$.

    In view of Lemma \eqref{lem est of q(x,z)},
    for the last term, it follows from
    $\abs{m(x,z)-m(x-h,z)}\leq C_R|z||h|$ for all $|x|\leq R$,
    given in the analysis for $R_5$.

    Now we turn to $R_7$. Note that
    \begin{align*}
       & m(x,z)+D_x\cdot p(x,z)-1 -
            [m(x-h,z)+D_x\cdot p(x-h,z)-1] \\
       & = m(x,z)+ (\text{div}_y p)(y(x,z),z)-1 -[m(x-h,z)+ (\text{div}_y p)(y(x-h,z),z)-1]\\
         & \quad   +[\text{div}_xp(x,z)- (\text{div}_y p)(y(x,z),z)]
            - [\text{div}_x p(x-h,z)- (\text{div}_y p)(y(x-h,z),z)]\\
       & =: R_7'+ R_7''
    \end{align*}
    Similar to $R_6$, with the estimate $\abs{D_x^2 p(x,z)}\leq C_R |z|$ from (H2'),
    \begin{align*}
      R_7'  = \frac{-P+P\tr M+(\tr M)^2}{1+\tr M+ P}(x,z)
            - \frac{-P+P\tr M+(\tr M)^2}{1+\tr M+ P}(x-h,z).
    \end{align*}
    can be estimated in the same manner and
    it is bounded by $C_R|z|^2|h|$.
    To illustrate, we treat one typical term in the above difference.
    \begin{align*}
       & (\partial_i p)(y(x-h,z),z)(\partial_j p)(y(x-h,z),z)
        -(\partial_i p)(y(x,z),z)(\partial_j p)(y(x,z),z)\\
       & = (\partial_i p)(y(x-h,z),z)\Big((\partial_j p)(y(x-h,z),z)-(\partial_j p)(y(x,z),z)\Big)\\
       & \quad +(\partial_j p)(y(x,z),z)\Big((\partial_i p)(y(x-h,z),z)-(\partial_i p)(y(x,z),z)\Big).
    \end{align*}
    By (H2') and Lemma \ref{lem bound of y(x,z)},
    \[
       \abs{(\partial_i p)(y(x,z),z)}\leq C(1+|y(x,z)|)|z|
        \leq 2C(1+|x|)|z|\leq C_R |z|.
    \]
    We also observe that
    \begin{align*}
       & (\partial_j p)(y(x-h,z),z)-(\partial_j p)(y(x,z),z) \\
       & =\int_0^1 (he_k-q(x-h,z)+q(x,z))\\
       & \qquad \cdot(D_x \partial_j p)
        \big(x-q(x,z)-\theta(he_k-q(x-h,z)+q(x,z)),z\big)d\theta.
    \end{align*}
    Then the estimate \eqref{eq est of Ri} will follow from
    (H2') and Lemma \ref{lem Lip q}.

    Next, we compute that
    \begin{align*}
      R_7'' & =\int_0^1 (q(x-h,z)-q(x,z))(D_x (\text{div}_x p))(x-\theta q(x,z),z)d\theta \\
       & \quad + \int_0^1\int_0^1 q(x-h,z)(h+\theta (q(x-h,z)-q(x,z)))\\
       & \qquad    \cdot\big( D_x^2(\text{div}_x p)\big)\big(x-h-\theta q(x-h,z)
            +\xi (h+\theta (q(x-h,z)-q(x,z))),z\big) d\xi.
    \end{align*}
    By (H2'),  we obtain
    $\abs{R_7''(x,z,h)}\leq C|z|^2 |h|$ for all $|z|<r$.

    For $R_8$, observe that
    \begin{align*}
      R_8= & \eta(x-h)\int_0^1\int_0^1 \theta|q(x,z)-q(x-h,z)|^2 \\
       &  \quad \cdot(D^2 v_h)\big(x-\xi\theta (q(x,z)-q(x-h,z)-h) \big)
            d\xi d\theta \,m(x,z).
    \end{align*}
    Therefore by Lemma \ref{lem Lip q},
    $\norm{\frac{1}{-h}\int_{|z|<r}R_8(\cdot,z,h)\nu(dz)}_{L^q(B_1)}
    \leq C\norm{\eta\Delta_h^k f}_{L^p (B_1)}^{p-1}$ too.
    The proof is complete.
\end{proof}


\bibliographystyle{amsplain}
\bibliography{Levy-PI}

\providecommand{\bysame}{\leavevmode\hbox to3em{\hrulefill}\thinspace}
\providecommand{\MR}{\relax\ifhmode\unskip\space\fi MR }
\providecommand{\MRhref}[2]{%
  \href{http://www.ams.org/mathscinet-getitem?mr=#1}{#2}
}
\providecommand{\href}[2]{#2}
\begin{thebibliography}{10}

\bibitem{abels2012pseudodifferential}
H.~Abels, \emph{{Pseudodifferential and Singular Integral Operators: An
  Introduction with Applications}}, De Gruyter Textbook, De Gruyter, 2012.

\bibitem{AlibaudCJ2012}
N.~Alibaud, S.~Cifani, and E.~R. Jakobsen, \emph{{Continuous Dependence
  Estimates for Nonlinear Fractional Convection-diffusion Equations}}, SIAM
  Journal on Mathematical Analysis \textbf{44} (2012), no.~2, 603--632.

\bibitem{applebaum2004levy}
D.~Applebaum, \emph{L{\'e}vy processes and stochastic calculus}, Cambridge
  Studies in Advanced Mathematics, Cambridge University Press, 2004.

\bibitem{Ar68}
D.~G. Aronson, \emph{Non-negative solutions of linear parabolic equations},
  Ann. Scuola Norm. Sup. Pisa (1968), no.~3, 607--694.

\bibitem{Bally2014}
V.~Bally and L.~Caramellino, \emph{{On the distances between probability
  density functions}}, Electron. J. Probab. \textbf{19} (2014), no.~110, 1--33.

\bibitem{Barndorff-Nielsen1997}
O.~E. Barndorff-Nielsen, \emph{{Normal Inverse Gaussian Distributions and
  Stochastic Volatility Modelling}}, Scand. J. Stat. \textbf{24} (1997), no.~1,
  1--13.

\bibitem{BL02}
R.~F. Bass and D.~A. Levin, \emph{Transition probabilities for symmetric jump
  processes}, Trans. Amer. Math. Soc. (2002), no.~7, 2933--2953.

\bibitem{Benth2001OptPort}
F.~E. Benth, K.~H. Karlsen, and K.~Reikvam, \emph{{Optimal portfolio management
  rules in a non-Gaussian market with durability and intertemporal
  substitution}}, Finance and Stochastics \textbf{5} (2001), no.~4, 447--467.

\bibitem{Bodnarchuk2008cond}
S.~V. Bodnarchuk and A.~M. Kulik, \emph{{Conditions for the existence and
  smoothness of the distribution density for the Ornstein-Uhlenbeck process
  with L\'evy noise}}, Teor. Imovir. Mat. Statyst. (2008), 21--35.

\bibitem{BKRS}
V.~I. Bogachev, N.~V. Krylov, M.~R\"ockner, and S.~V. Shaposhnikov,
  \emph{Fokker-{P}lanck-{K}olmogorov equations}, vol. 207, American
  Mathematical Society, Providence, RI, 2015.

\bibitem{BJ07}
K.~Bogdan and T.~Jakubowski, \emph{Estimates of heat kernel of fractional
  laplacian perturbed by gradient opera- tors}, Comm. Math. Phys. (2007),
  no.~1, 179--198.

\bibitem{Bony1967}
J.-M. Bony, \emph{{Principe du maximum dans les espaces de Sobolev}}, Comptes
  Rendus Acad. Sci. Paris, S\'erie A \textbf{265} (1967), 333--336 (French).

\bibitem{Brooks2006}
B.~P. Brooks, \emph{{The coefficients of the characteristic polynomial in terms
  of the eigenvalues and the elements of an $n\times n$ matrix}}, Appl. Math.
  Lett. \textbf{19} (2006), no.~6, 511--515.

\bibitem{Butko2016}
Y.~A. Butko, M.~Grothaus, and O.~G. Smolyanov, \emph{Feynman formulae and phase
  space {F}eynman path integrals for tau-quantization of some
  {L}\'evy-{K}hintchine type {H}amilton functions}, J. Math. Phys. \textbf{57}
  (2016), no.~2, 023508, 22.

\bibitem{CaiYang2014}
J.~Cai and H.~Yang, \emph{On the decomposition of the absolute ruin probability
  in a perturbed compound poisson surplus process with debit interest}, Annals
  of Operations Research \textbf{212} (2014), no.~1, 61--77.

\bibitem{Cass20091416}
T.~Cass, \emph{{Smooth densities for solutions to stochastic differential
  equations with jumps}}, Stochastic Processes and their Applications
  \textbf{119} (2009), no.~5, 1416--1435.

\bibitem{LC2016}
L.~Chen, \emph{The numerical path integration method for stochastic
  differential equations}, Ph.D. thesis, NTNU Norwegian University of Science
  and Technology, 2016.

\bibitem{CNJLocal}
L.~Chen, E.~R. Jakobsen, and A.~Naess, \emph{On numerical density
  approximations of solutions of {SDE}s with unbounded coefficients}, Adv.
  Comput. Math. Online First (2017), 1--29.

\bibitem{CHXZ17}
Z.-Q. Chen, E.~Hu, L~Xie, and X.~Zhang, \emph{Heat kernels for non-symmetric
  diffusion operators with jumps}, J. Differential Equations \textbf{263}
  (2017), no.~10, 6576--6634.

\bibitem{CK10}
Z.-Q. Chen and T.~Kumagai, \emph{A priori h\"older estimate, parabolic harnack
  principle and heat kernel estimates for diffusions with jumps.}, Rev. Mat.
  Iberoam. (2010), no.~2, 551--589.

\bibitem{tankov2003financial}
R.~Cont and P.~Tankov, \emph{Financial modelling with jump processes}, Chapman
  and Hall/CRC Financial Mathematics Series, CRC Press, 2003.

\bibitem{schwartz1958linear}
N.~Dunford, J.T. Schwartz, W.G. Bade, and R.G. Bartle, \emph{{Linear Operators,
  Part I: General Theory}}, Pure and Applied Mathematics. vol. 7, etc, New
  York; Groningen printed, 1958.

\bibitem{engel2006one}
K.~J. Engel, R.~Nagel, et~al., \emph{One-parameter semigroups for linear
  evolution equations}, Graduate Texts in Mathematics, Springer New York, 2006.

\bibitem{evans2010partial}
L.~C. Evans, \emph{Partial differential equations}, Graduate studies in
  mathematics, American Mathematical Society, 2010.

\bibitem{Fornaro2007747}
S.~Fornaro and L.~Lorenzi, \emph{{Generation results for elliptic operators
  with unbounded diffusion coefficients in Lp- and Cb-spaces}}, Discrete and
  Continuous Dynamical Systems \textbf{18} (2007), no.~4, 747--772.

\bibitem{garroni2002second}
M.~G. Garroni and J.~L. Menaldi, \emph{Second order elliptic
  integro-differential problems}, Chapman \& Hall/CRC Research Notes in
  Mathematics Series, CRC Press, 2002.

\bibitem{GS1972}
{\u{I}}.~{\=I}. G{\={\i}}hman and A.~V. Skorohod, \emph{Stochastic differential
  equations}, Springer-Verlag, New York-Heidelberg, 1972, Translated from the
  Russian by Kenneth Wickwire, Ergebnisse der Mathematik und ihrer
  Grenzgebiete, Band 72.

\bibitem{gilbarg2015elliptic}
D.~Gilbarg and N.S. Trudinger, \emph{Elliptic partial differential equations of
  second order}, Classics in Mathematics, Springer Berlin Heidelberg, 2015.

\bibitem{GimbertLions1984}
F.~Gimbert and P.-L. Lions, \emph{Existence and regularity results for
  solutions of second-order elliptic integro-differential operators}, Ricerche
  di Matematica \textbf{33} (1984), no.~2, 315--358.

\bibitem{hiraba1992}
S.~Hiraba, \emph{{Existence and smoothness of transition density for jump-type
  Markov processes: applications of Malliavin calculus}}, Kodai Math. J.
  \textbf{15} (1992), no.~1, 28--49.

\bibitem{Kolokoltsov2000}
V.~Kolokoltsov, \emph{{Symmetric stable laws and stable-like jump-diffusions}},
  Proc. London Math. Soc. \textbf{80} (2000), no.~03.

\bibitem{Komatsu2001}
T.~Komatsu and A.~Takeuchi, \emph{{On the Smoothness of PDF of Solutions to SDE
  of Jump Type}}, {International Journal of Differential Equations and
  Applications} \textbf{2} (2001), no.~2, 141--197.

\bibitem{kyprianou2006exotic}
A.~Kyprianou, W.~Schoutens, and P.~Wilmott, \emph{{Exotic Option Pricing and
  Advanced L\'evy Models}}, Wiley, 2006.

\bibitem{Lions1983}
P.-L. Lions, \emph{{A remark on Bony maximum principle}}, Proc. Amer. Math.
  Soc. \textbf{88} (1983), 503--508.

\bibitem{mikosch2009non}
T.~Mikosch, \emph{Non-life insurance mathematics: An introduction with the
  poisson process}, Universitext, Springer Berlin Heidelberg, 2009.

\bibitem{Oksendal2007}
B.~{\O}ksendal and A.~Sulem, \emph{{Applied Stochastic Control of Jump
  Diffusions}}, Universitext, Springer Berlin Heidelberg, 2007.

\bibitem{pazy1992semigroups}
A.~Pazy, \emph{Semigroups of linear operators and applications to partial
  differential equations}, Applied Mathematical Sciences, Springer New York,
  1992.

\bibitem{price1951}
G.~B. Price, \emph{{Bounds for Determinants with Dominant Principal Diagonal}},
  Proceedings of the American Mathematical Society \textbf{2} (1951), no.~3,
  497--502.

\bibitem{protter2005stochastic}
P.~Protter, \emph{Stochastic integration and differential equations},
  Stochastic Modelling and Applied Probability, Springer Berlin Heidelberg,
  2005.

\bibitem{schechter2001principles}
M.~Schechter, \emph{Principles of functional analysis}, Graduate studies in
  mathematics, American Mathematical Society, 2001.

\bibitem{schoutens2003levy}
W.~Schoutens, \emph{{L\'evy Processes in Finance: Pricing Financial
  Derivatives}}, Wiley Series in Probability and Statistics, Wiley, 2003.

\bibitem{Wang2013}
J.~Wang, \emph{{Sub-Markovian $C_0$-Semigroups Generated by Fractional
  Laplacian with Gradient Perturbation}}, Integral Equations and Operator
  Theory \textbf{76} (2013), no.~2, 151--161.

\bibitem{wheeden1977measure}
R.~L. Wheeden and A.~Zygmund, \emph{Measure and integral: An introduction to
  real analysis}, Chapman \& Hall/CRC Pure and Applied Mathematics, Taylor \&
  Francis, 1977.

\bibitem{Zhang20121867}
W.~Zhang and J.~Bao, \emph{{Regularity of Very Weak Solutions for Elliptic
  Equation of Divergence Form}}, Journal of Functional Analysis \textbf{262}
  (2012), no.~4, 1867 -- 1878.

\bibitem{ZhangBao2013}
\bysame, \emph{{Regularity of Very Weak Solutions for Nonhomogeneous Elliptic
  Equation}}, Communications in Contemporary Mathematics \textbf{15} (2013),
  no.~04, 1350012.

\bibitem{zhang2014}
X.~Zhang, \emph{{Densities for SDEs driven by degenerate $\alpha$-stable
  processes}}, Annals of Probability \textbf{42} (2014), no.~5, 1885--1910.

\end{thebibliography}

\end{document}